\documentclass{article}
\usepackage{tikz}
\usetikzlibrary{arrows}
\usepackage[framemethod=tikz]{mdframed}
\usepackage{wrapfig}
\usepackage{amsmath,amssymb}
\usepackage{type1cm}
\usepackage{amsthm}
\usepackage{mathrsfs}
\usepackage{mathtools}
\usepackage{enumerate}
\AtBeginDocument{%
   \def\MR#1{}
}

\theoremstyle{definition}
\newtheorem{theo}{Theorem}[section]
\newtheorem{prop}[theo]{Proposition}
\newtheorem{defi}[theo]{Definition}
\newtheorem{lemm}[theo]{Lemma}
\newtheorem{coro}[theo]{Corollary}

\newtheorem{rema}[theo]{Remark}

\begin{document}
\title{Geometric Structure of Affine Deligne-Lusztig Varieties for $\mathrm{GL}_3$}
\author{Ryosuke Shimada}
\date{}
\maketitle

\begin{abstract}
In this paper we study the geometric structure of affine Deligne-Lusztig varieties $X_{\lambda}(b)$ for $\mathrm{GL}_3$ and $b$ basic.
We completely determine the irreducible components of the affine Deligne-Lusztig variety.
In particular, we classify the cases where all of the irreducible components are classical Deligne-Lusztig varieties times finite-dimensional affine spaces.
If this is the case, then the irreducible components are pairwise disjoint. 
\end{abstract}

\section{Introduction}
Let $k$ be a field with $q$ elements, and let $\bar{k}$ be an algebraic closure of $k$.
Let $F=k((t))$ and let $L =\bar{k}((t))$.
We write $\mathcal O=\bar{k}[[t]], \mathcal O_F=k[[t]]$ for the valuation rings of $L$ and $F$.
Let $\sigma$ denote the Frobenius morphism of $\bar{k}/k$ and also of $L/F$.

Let $G$ be a split connected reductive group over $k$ and let $T$ be a split maximal torus of it.
Let $B$ be a Borel subgroup of $G$ containing $T$. 
For a cocharacter $\lambda\in X_*(T)$, let $t^{\lambda}$ be the image of $t\in \mathbb G_m(F)$ under the homomorphism $\lambda\colon\mathbb G_m\rightarrow T$.

We fix a dominant cocharacter $\lambda\in X_*(T)$.
Then the affine Deligne-Lusztig variety $X_{\lambda}(b)$ is the locally closed reduced $\bar k$-subscheme of the affine Grassmannian defined as
$$X_{\lambda}(b)(\bar{k})=\{xG(\mathcal O)\in G(L)/G(\mathcal O)\mid x^{-1}b\sigma(x)\in G(\mathcal O)t^{\lambda}G(\mathcal O)\}.$$
Analogously, we can also define the affine Deligne-Lusztig varieties associated to arbitrary parahoric subgroups (especially Iwahori subgroups).

The affine Deligne-Lusztig variety $X_{\lambda}(b)$ carries a natural action (by left multiplication) by the group
$$J_b=J_b(F)=\{g\in G(L)\mid g^{-1}b\sigma(g)=b\}.$$
This action induces an action of $J_b$ on the set of irreducible components.

The geometric properties of affine Deligne-Lusztig varieties have been studied by many people.
For example, there is a simple criterion by Kottwitz and Rapoport to decide whether an affine Deligne-Lusztig variety is non-empty (see \cite{Gashi}).
Moreover, for $X_{\lambda}(b)\neq \emptyset$, we have an explicit dimension formula:
$$\mathop{\mathrm{dim}} X_{\lambda}(b)=\langle\rho, \lambda-\nu_b \rangle-\frac{1}{2}\mathrm{def}(b),$$
where $\nu_b$ is the Newton vector of $b$, $\rho$ is half the sum of the positive roots, and $\mathrm{def}(b)$ is the defect of $b$.
For split groups, the formula was obtained in \cite{GHKR} and \cite{Viehmann}.
Recently, the parametrization problem of top-dimensional irreducible components of $X_{\lambda}(b)$ was also solved.
See \cite{Nie} and \cite{ZZ}.

Besides the geometric properties as above, it is known that in certain cases, the (closed) affine Deligne-Lusztig variety admits a simple description.
Let $P$ be a standard rational maximal parahoric subgroup of $G(L)$, where $G$ is assumed to be simple.
Then for any minuscule cocharacter $\mu$ and for $b\in G(L)$,
G\"{o}rtz and He in \cite{GH} (see also \cite{GH2}) studied the following union of affine Deligne-Lusztig varieties:
$$X(\mu, b)_P=\{g\in G(L)/P\mid g^{-1}b\sigma (g)\in \bigcup_{w\in\mathrm{Adm}(\mu)}PwP\}.$$
They proved that if $(G, \mu, P)$ is of Coxeter type and if $b$ lies inside the basic $\sigma$-conjugacy class, then $X(\mu, b)_P$ is naturally a union of classical Deligne-Lusztig varieties.
Furthermore, the work \cite{GHN} studied the generalized affine Deligne-Lusztig variety $X(\mu, b)_{K'}$ associated to a standard parahoric subgroup $K'$.
The main result in \cite{GHN} determines when $X(\mu, b)_{K'}$ is naturally a union of classical Deligne-Lusztig varieties.
In particular, the existence of such a simple description is independent of $K'$.

In the Iwahori case, Chan and Ivanov \cite{CI} gave an explicit description of certain Iwahori-level affine Deligne-Lusztig varieties for $\mathrm{GL}_n$.
Each component of the disjoint decomposition described there is a classical Deligne-Lusztig variety times finite-dimensional affine space, and they point out the similarity between their description and the results in \cite{GH}.

In this paper, we study the geometry of affine Deligne-Lusztig varieties $X_{\lambda}(b)$ in the affine Grassmannian  for $G=\mathrm{GL}_3$ and $b$ basic.
For this, it is enough to consider the case for $b$ whose newton vector is of the form $(\frac{i}{3}, \frac{i}{3}, \frac{i}{3})$ ($i=0,1,2$).
For any dominant $\lambda\in X_*(T)$ with $X_{\lambda}(b)\neq \emptyset$, we completely determine the irreducible components of $X_{\lambda}(b)$ and the index set of them.
\begin{theo}
The irreducible components of $X_{\lambda}(b)$ are parameterized by the set $\bigsqcup_{\mu\in M}J_b/K_{\mu}$, where $M$ is a finite set consisting of certain dominant cocharacters $\mu$ determined by $\lambda$, and $K_{\mu}$ is the stabilizer of a lattice depending on $\mu$ under the action of $J_b$.
Each irreducible component is an affine bundle over a simple variety.
\end{theo}
The group $J_b$ acts on the index set by left multiplication.
In most cases, the irreducible component is an open subscheme of a finite-dimensional affine space, which is also affine.
For more details, see Theorem \ref{ADLVgeometricstructure} and Theorem \ref{ADLVgeometricstructure2}.

As an immediate corollary of this result, we classify the cases where all of the irreducible components of $X_{\lambda}(b)$ are classical Deligne-Lusztig varieties times finite-dimensional affine spaces (Corollary \ref{ADLVgeometricstructure3}).
In such cases, $X_{\lambda}(b)$ is a disjoint union of the irreducible components.
\begin{theo}
Let $\lambda\neq 0$ be as in Corollary \ref{ADLVgeometricstructure3}, and let $\Omega$ be the Drinfeld upper half space over $k$ of dimension $2$.
\begin{enumerate}[(i)]
\item If $b=1$, then we have
\begin{align*}
X_{\lambda}(1)&\cong \bigsqcup_{J_1/K_0}\Omega\times \mathbb A \hspace{3.4cm} \text{or}\\
X_{\lambda}(1)&\cong (\bigsqcup_{J_1/K_1}\Omega\times \mathbb A)\sqcup (\bigsqcup_{J_1/K_2}\Omega\times \mathbb A)
\end{align*}
as $\bar k$-varieties, where $\mathbb A$ is a finite-dimensional affine space and $K_i$ is the stabilizer of a lattice under the action of $J_1$.
\item If $b$ has the newton vector of the form $(\frac{i}{3}, \frac{i}{3}, \frac{i}{3})$ ($i=1,2$), then we have
\begin{align*}
X_{\lambda}(b)&\cong \bigsqcup_{J_b/H_b}\mathbb A \hspace{2.7cm} \text{or}\\
X_{\lambda}(b)&\cong (\bigsqcup_{J_b/H_b}\mathbb A)\sqcup (\bigsqcup_{J_b/H_b}\mathbb A)
\end{align*}
as $\bar k$-varieties, where $\mathbb A$ is a finite-dimensional affine space and $H_b$ is the stabilizer of a lattice under the action of $J_b$.
\end{enumerate}
\end{theo}

The strategy of the proof is as follows:
Using the Bruhat-Tits building of $\mathrm{SL_3}$, we decompose $X_{\lambda}(b)(\bar k)$ into closed subsets, which are actually irreducible components.
This can be checked using the dimension formula above.
Then we determine their geometric structure by embedding them into the Schubert cells.
The crucial ingredient of these processes is the method of Kottwitz in \cite{Kottwitz}.
\\

\textbf{Acknowledgments:} This paper is the author's master's thesis, written at the University of Tokyo.
The author is grateful to his advisor Yoichi Mieda for his encouragement and helpful comments.

\section{Definition and Basic Properties}
Let $k$ be a field with $q$ elements, and let $\bar{k}$ be an algebraic closure of $k$.
We set $F=k((t))$ and $L =\bar{k}((t))$.
Further, we write $\mathcal O=\bar{k}[[t]], \mathcal O_F=k[[t]]$ for the valuation rings of $L$ and $F$.
Let $\sigma$ denote the Frobenius morphism of $\bar{k}/k$. 
We extend $\sigma$ to the Frobenius morphism of $L/F$, i.e., $\sigma(\sum a_n t^n)=\sum\sigma(a_n)t^n$.

\subsection{The Affine Grassmannian}
In this subsection, we define the affine Grassmannian of $\mathrm{SL}_3$ and $\mathrm{GL}_3$, cf.\ \cite{Gortz}.
Set $G=\mathrm{SL}_3$ or $\mathrm{GL}_3$.
The loop group $LG$ of $G$ is the $k$-space given by the following functor:$$LG(R)=G(R((t))),$$ where $R$ is a $k$-algebra.
Similarly, we have the positive loop group $L^{+}G$, defined as $$L^{+}G(R)=G(R[[t]]),$$ where $R$ is a $k$-algebra.
The positive loop group is actually an (infinite-dimensional) scheme. 
\begin{defi}
The {\it affine Grassmanian} $\mathcal{G}rass_{G}$ for $G$ is the quotient $k$-space $LG/L^{+}G$.
\end{defi}
The quotient in the category of $k$-spaces is the sheafification of the presheaf quotient $R\mapsto LG(R)/L^{+}G(R)$.

The affine Grassmannian is an ind-scheme.
To check this, we define the notion of lattices.
Let $R$ be a $k$-algebra, and let $r \in \mathbb Z_{>0}$. The $R[[t]]$-submodule $R[[t]]^{3}\subset R((t))^{3}$ is called the {\it standard lattice} and denoted by $\Lambda_{R}$.

\begin{defi}
\label{latticedefi}
A {\it lattice} $\mathscr L\subset R((t))^{3}$ is a $R[[t]]$-module such that
\begin{enumerate}[(i)]
\item there exists $N \in \mathbb Z_{\geq 0}$ with $$t^{N}\Lambda_{R}\subseteq \mathscr L \subseteq t^{-N}\Lambda_{R},\ \mathrm{and}$$

\item the quotient $t^{-N}\Lambda_{R}/\mathscr L$ is locally free of finite rank over $R$.
\end{enumerate}
For $r\in \mathbb Z$, a lattice $\mathscr L$ is said to be $r$-{\it special} if $\bigwedge^{3}\mathscr L=t^{r}\Lambda_{R}$.
\end{defi}

We denote the set of all lattices in $R((t))^{3}$ by $\mathcal{L}att(R)$, and the set of all $0$-special lattices by $\mathcal{L}att^{0}(R)$.
We also define, for $N\geq 1$, subsets 
$$\mathcal{L}att^{(N)}(R)\subset \mathcal{L}att(R),\ \mathcal{L}att^{0,(N)}(R)\subset \mathcal{L}att^{0}(R),$$
where the number $N$ in part (i) of the definition of a lattice is fixed.
For $\mathcal{L}att^{0,(N)}$, the morphism of functors $$\mathcal{L}att^{0,(N)}(R)\rightarrow {\mathop{\mathrm{Grass}}}_{3N}(t^{-N}\Lambda_k/t^{N}\Lambda_k)(R),\mathscr L\mapsto \mathscr L/t^{N}\Lambda_R,$$
defines a closed embedding of $\mathcal{L}att^{0,(N)}$ 
into the Grassmann variety of $3N$-dimensional subspaces of the $6N$-dimensional $k$-vector space $t^{-N}\Lambda_k/t^{N}\Lambda_k$.
Thus we have a filtration of $\mathcal{L}att^{0}$ by closed subschemes:
$$\mathcal{L}att^{0}=\bigcup\mathcal{L}att^{0,(N)}.$$
This gives an ind-projective scheme structure of $\mathcal{L}att^{0}$ because each $\mathcal{L}att^{0,(N)}$ is a closed subscheme of the projective scheme.
Similarly, we can show that each subfunctor $\mathcal{L}att^{(N)}$ has a projective scheme structure.
So the filtration $$\mathcal{L}att=\bigcup\mathcal{L}att^{(N)}$$
gives an ind-projective structure of $\mathcal{L}att$.
The ind-scheme $\mathcal{L}att^{0}$ is integral.
See \cite[Propostion 6.4]{BL} and \cite[Theorem 6.1]{PR}, which include the case of positive characteristic and also deal with other groups.
On the other hand, $\mathcal{L}att$ is not reduced in general (cf.\  \cite[Example 2.9]{Gortz} for the case $\mathbb G_m$). 
However, we have $$(\mathcal{L}att)_{red}=\bigsqcup_{r\in \mathbb Z} \mathcal{L}att^{r},$$ where $\mathcal{L}att^{r}$ is the $k$-subspace of $\mathcal{L}att$ consisting of $r$-special lattices.

Finally, note that the affine Grassmannian for $\mathrm{GL}_3$ (resp.\ $\mathrm{SL}_3$) is naturally isomorphic, as a $k$-space, to $\mathcal{L}att$ (resp.\ $\mathcal{L}att^{0}$).
So the affine Grassmannian has an ind-scheme structure.
From now on, we write $X=\mathcal{G}rass_{\mathrm{GL}_3}, X^S=\mathcal{G}rass_{\mathrm{SL}_3}$.
Since $\mathcal{L}att^0$ is isomorphic to $\mathcal{L}att^r$ through left multiplication by $a_r\in \mathrm{GL}_n(k((t)))$ with $v(\det (a_r))=r$, we have an isomorphism of ind-schemes 
$$X_{red}\cong \bigsqcup_{r\in \mathbb Z} X^S$$
depending on the choice of $(a_r)_r$.

\subsection{Affine Deligne-Lusztig Varieties in the Affine Grassmannian}
We keep the notation above.
Moreover, we fix a maximal torus $T$ and a Borel subgroup $B$ of $G$.
In the case $G=\mathrm{GL}_3$, we always let $T$ be the torus of diagonal matrices, and we choose the subgroup of upper triangular matrices $B$ as a Borel subgroup.
For $G=\mathrm{SL}_3$, we make analogous choices of a maximal torus and a Borel subgroup.
We often denote by $T^S$ and $B^S$ the maximal torus and the Borel subgroup of $\mathrm{SL}_3$ to avoid confusion with those of $\mathrm{GL}_3$.
Finally we set $K=\mathrm{GL}_3(\mathcal O),K^S=\mathrm{SL}_3(\mathcal O)$.

We denote by $\Phi$ the set of roots given by the choice of $T$, and by $\Phi_+$ the set of positive roots distinguished by $B$.
We let $$X_*(T)_+=\{\lambda\in X_*(T)\mid \langle \alpha, \lambda \rangle\geq0\ \text{for all}\ \alpha\in \Phi_+\}$$ denote the set of dominant cocharacters.
Then the Cartan decomposition of $G(L)$ is given by 
$$G(L)=\bigcup_{\lambda\in X_*(T)_+}G(\mathcal O)t^{\lambda}G(\mathcal O).$$
We illustrate this with $G=\mathrm{GL}_3, \mathrm{SL}_3$. 
For $G=\mathrm{GL}_3$, let $\chi_{ij}$ be the character $T\rightarrow \mathbb G_m$ defined by $\mathrm{diag}(t_1,t_2,t_3)\mapsto t_i{t_j}^{-1}$.
Then we have $\Phi=\{\chi_{ij}\mid i\neq j\}, \Phi_+=\{\chi_{ij}\mid i< j\}$ with respect to our choice of $T$ and $B$. So the subset ${X_*(T)}_+\subset X_*(T)\cong \mathbb Z^3$ is equal to the set $\{(m_1,m_2, m_3)\in \mathbb Z^3\mid m_1\geq m_2 \geq m_3\}$ and thus we have $$\{t^{\lambda}\mid \lambda \in X_*(T)_+\}=\{\mathrm{diag}(t^{m_1},t^{m_2},t^{m_3})\in T\mid m_1\geq m_2\geq m_3, m_i\in \mathbb Z\}.$$
Similarly, for $G=\mathrm{SL}_3$, the set $\{t^{\lambda}\mid \lambda \in X_*(T)_+\}$ is equal to the set $$\{\mathrm{diag}(t^{m_1},t^{m_2},t^{m_3})\in T\mid m_1\geq m_2 \geq m_3, \sum_im_i=0, m_i\in \mathbb Z\}.$$
For $\lambda=(m_1,m_2, m_3),\mu=(m_1',m_2',m_3')\in X_*(T)_+$, we write $\lambda\preceq \mu$ if $m_1\le m_1', m_1+m_2\le m_1'+m_2', m_1+m_2+m_{3}=m_1'+m_2'+m_{3}'$.
This is called the dominance order.

Given $b\in G(L)$, its $\sigma$-{\it stabilizer} is the group $$\{g\in G(L)\mid  g^{-1}b\sigma(g)=b\}.$$
We denote by $J_b$ (resp.\ $J_b^S$) the $\sigma$-stabilizer for $\mathrm{GL}_3$ (resp.\ $\mathrm{SL}_3$).
\begin{lemm}
\label{Jb}
Let $b\in \mathrm{GL}_3(L)$. 
The restriction of $\eta=v_L\circ \det\colon\mathrm{GL}_3(L)\rightarrow \mathbb Z$ to $J_b$ is surjective.
\end{lemm}
\begin{proof}
Since conjugation does not change the determinant, it is enough to prove the statement for some representatives of the $\sigma$-conjugacy classes of $\mathrm{GL}_3(L)$.
As described in \cite[Example 4.6]{Gortz}, every $\sigma$-conjugacy class in $\mathrm{GL}_3(L)$ contains a representative $b$ of the following form: $b$ is a block diagonal matrix, and each block has the form
$$
\begin{pmatrix}
0 & t^{a_i+1}E_{k_i} \\
t^{a_i}E_{n_i-k_i} & 0 \\
\end{pmatrix}
\in \mathrm{GL}_{n_i}(L)
$$
Here $3=\sum_i n_i, a_i, k_i\in \mathbb Z, 0\le k_i <n_i$.
Using this description, we can easily find $g\in J_b$ such that $\eta(g)=r$ for any $r\in \mathbb Z$ and $b$.
\end{proof}
For $b\in \mathrm{GL}_3(L)$ set $H_b=\mathop{\mathrm{Ker}}(v_L\circ \det\colon J_b\rightarrow \mathbb Z)$.
Then $v_L\circ \det$ induces an isomorphism $J_b/H_b\cong \mathbb Z$ and if $b\in \mathrm{SL}_3(L)$, we clearly have $J_b^S\subseteq H_b$.

\begin{defi}
The {\it relative position map} is $$\mathrm{inv}\colon G(L)/G(\mathcal O)\times G(L)/G(\mathcal O)\rightarrow X_*(T)_+,$$
which maps a pair of cosets $(xG(\mathcal O), yG(\mathcal O))$ to the unique element $\lambda \in X_*(T)_+$ such that $x^{-1}y\in G(\mathcal O)t^{\lambda} G(\mathcal O)$.
\end{defi}

We now come to the definition of the affine Deligne-Lusztig variety.
\begin{defi}
The {\it affine Deligne-Lusztig variety} $X_{\lambda}(b)$ in the affine Grassmannian associated with $b\in G(L)$ and $\lambda\in {X_*(T)}_+$ is the locally closed subset of $\mathcal{G}rass_G$ given by $$X_{\lambda}(b)(\bar{k})=\{xG(\mathcal O)\in G(L)/G(\mathcal O)\mid x^{-1}b\sigma(x)\in G(\mathcal O)t^{\lambda}G(\mathcal O)\},$$
provided with the reduced sub-ind-scheme structure.
\end{defi}

In fact, $X_{\lambda}(b)$ is a scheme locally of finite type over $\bar k$; see \cite[Corollary 6.5]{HV}.

From now on, $X_{\lambda}(b)$ (resp.\ $X_{\lambda}^S(b)$) always denotes the affine Deligne-Lusztig variety for $\mathrm{GL}_3$ (resp.\ $\mathrm{SL}_3$).
Then $J_b$ (resp.\ $J_b^S$) acts by left multiplication on $X_{\lambda}(b)$ (resp.\ $X_{\lambda}^S(b)$).
For affine Deligne-Lusztig varieties, we have the following basic lemma.

\begin{lemm}
\label{ADLVbasic}
Let $\lambda=(m_1,m_2,m_3)\in X_*(T)_+$ and $b,g\in \mathrm{GL}_3(L)$.

\begin{enumerate}[(i)]
\item If $X_{\lambda}(b)$ is non-empty, then $v_L(\det(b))=m_1+m_2+m_3$.

\item The varieties $X_{\lambda}(b)$ and $X_{\lambda}(g^{-1}b\sigma(g))$ are isomorphic.

\item Let $c=\mathrm{diag}(t^m,t^m,t^m), m\in \mathbb Z$. Then $X_{\lambda}(b)$ and $X_{\lambda+M}(cb)$ are equal as subvarieties of $X$, where $M=(m, m, m)$.
\end{enumerate}
\end{lemm}

\begin{proof}
(i) follows from the equality $v_L(\det(x^{-1}b\sigma(x)))=v_L(\det(b))$.

(ii) The map $x\mapsto g^{-1}x$ gives an isomorphism.

(iii) We have $xK\in X_{\lambda}(b) \Leftrightarrow x^{-1}b\sigma(x)\in Kt^{\lambda}K \Leftrightarrow x^{-1}cb\sigma(x)\in Kt^{\lambda+M}K\Leftrightarrow xK\in X_{\lambda+M}(cb)$.
\end{proof}

Next proposition gives a decomposition of the affine Deligne-Lusztig variety corresponding to the decomposition of the affine Grassmannian  $X_{red}\cong \bigsqcup_r X^S$.
\begin{prop}
\label{ADLVdecomposition}
Set $\eta=v_L\circ \det\colon \mathrm{GL}_3(L)\rightarrow \mathbb Z$.
Then we have $$X_{\lambda}(b)\cong \bigsqcup_{J_b/H_b} (X_{\lambda}(b)\cap \eta^{-1}(0))$$ as $\bar{k}$-varieties, and $J_b$ acts on the set of these components by left multiplication.
\end{prop}
\begin{proof}
The scheme structure on $X_{\lambda}(b)$ is the reduced one, thus the inclusion $X_{\lambda}(b)\subset X$ factors through $X_{red}\rightarrow X$.
By Lemma \ref{Jb}, we can choose $(a_r)_{r\in \mathbb Z}$ with $a_0=1, a_r\in J_b$ and $v_L(\det(a_r))=r$.
Then the isomorphism $X_{red}\cong \bigsqcup_r X^S$ determined by $(a_r)_r$ restricts to the isomorphism
$$X_{\lambda}(b)\cong \bigsqcup_{J_b/H_b} (X_{\lambda}(b)\cap \eta^{-1}(0))$$
as $\bar{k}$-varieties.
\end{proof}

\begin{rema}
Let $b=1$ and let $\lambda\in X_*(T)_+$ with $X_{\lambda}(1)\neq \emptyset$.
Then $X_{\lambda}(1)\cap \eta^{-1}(0)$ is equal to $X_{\lambda}^S(1)$.
\end{rema}

In this paper, we will treat the affine Deligne-Lusztig variety $X_{\lambda}(b)$ with $b$ basic (i.e. its newton vector $\nu_b$ is central).
Then by Lemma \ref{ADLVbasic}, it is enough to consider the following three cases:
\begin{enumerate}[(i)]
\item $b=1$;

\item $b=b_1\coloneqq
\begin{pmatrix}
0 & 0 & t \\
1 & 0 & 0 \\
0 & 1 & 0\\
\end{pmatrix}$;

\item$b=b_2\coloneqq
\begin{pmatrix}
0 & t & 0 \\
0 & 0 & t \\
1 & 0 & 0 \\
\end{pmatrix}$.
\end{enumerate}

\section{The Bruhat-Tits Building}
\label{building}

\subsection{The Affine Building of $\mathrm{SL}_3$}
In this subsection, we recall the Bruhat-Tits buildings of the groups  $\mathrm{SL}_3(L)$ and $\mathrm{SL}_3(F)$.
Here we will discuss mainly over $L$, and the same construction and definitions also work over $F$.

\begin{defi}
The {\it Bruhat-Tits Building} of $\mathrm{SL}_3(L)$ is the simplicial complex $\mathcal B_{\infty}$ such that
\begin{enumerate}[(i)]
\item The set of vertices of $\mathcal B_{\infty}$ is the set of equivalence classes of $\mathcal O$-lattices $\mathscr{L}\subseteq L^3$, where the equivalence relation is given by homothety, i.e., $\mathscr{L}\sim \mathscr{L}'$ if and only if there exists $c\in L^\times$ such that $\mathscr{L}'=c\mathscr{L}$.

\item A set $\{L_1,\ldots,L_m\}$ of $m$ vertices is a simplex if and only if there exist representatives $\mathscr{L}_i$ of $L_i$ such that $$\mathscr{L}_1\supset \cdots \supset \mathscr{L}_m \supset t\mathscr{L}_1.$$
\end{enumerate}
\end{defi}

We denote the simplicial complex arising from this construction over $F$ by $\mathcal B_1$.
We see $\mathcal B_1$ as a subset of $\mathcal B_{\infty}$ by sending an $\mathcal O_F$-lattice $\mathscr L\subseteq F^3$ to $\mathscr{L}\otimes_{\mathcal O_F}\mathcal O\subseteq L^3$.
Maximal simplices are called {\it alcoves} or {\it chambers}.
For a vertex of $\mathcal B_{\infty}$ represented by $\mathscr{L}=g\Lambda_{\bar k}$ with $g\in \mathrm{GL}_3(L)$, we call the residue class of $v_L(\det(g))$ in $\mathbb Z/3$ the {\it type} of the vertex.
This number is independent of the choice of a representative.

We say that an $\mathcal O$-lattice $\mathscr L\subset L^3$ is {\it adapted} to a basis $f_1,f_2, f_3$ of $L^3$, if $\mathscr L$ has an $\mathcal O$-basis of the form $t^{i_1}f_1,t^{i_2}f_2,t^{i_3}f_3$.
The {\it apartment} corresponding to the basis $f_i$ is the subcomplex of $\mathcal B_{\infty}$ whose simplices consist of vertices given by lattices adapted to this basis.
The apartment corresponding to the standard basis is called the {\it main apartment} and denoted by $\mathcal A_M$.

The action of $\mathrm{SL}_3(L)$ (resp.\ $\mathrm{SL}_3(F)$) on all vertices with the same type in $\mathcal B_{\infty}$ (resp.\ $\mathcal B_1$) is transitive. 
One has a base vertex of type $0$ represented by $\Lambda_{\bar k}$ (resp.\ $\Lambda_k$) with stabilizer $\mathrm{SL}_3(\mathcal O)$ (resp.\ $\mathrm{SL}_3(\mathcal O_F)$).
Thus we have
\begin{align*}
  X^S(k) &=  \text{the set of all vertices of type $0$ in $\mathcal B_1$} \\
  X^S(\bar k) &=  \text{the set of all vertices of type $0$ in $\mathcal B_{\infty}$}.
\end{align*}
The Frobenius morphism $\sigma$ acts on $\mathcal B_{\infty}$ in the obvious way.
Then the full subcomplex consisting of all vertices fixed by $\sigma$ is $\mathcal B_1$.

\subsection{The Relative Position}
In this subsection, our goal is to describe the formula on the relative position of two vertices in $\mathcal B_{\infty}$.
For this, we mainly refer to \cite{Kottwitz} (and especially to the first three sections).
Moreover, see \cite{Garrett} for the general notion on buildings.

We begin by introducing an equivalence relation $\sim$ on $X_*(T)_+\subset \mathbb Z^3$.
We will say that $(m_1, m_2, m_3)$ and $(m_1',m_2',m_3')$ are equivalent (and write $(m_1, m_2, m_3)\sim (m_1',m_2',m_3')$) if there exists an integer $m$ such that $(m_1, m_2, m_3)=(m_1'+m,m_2'+m,m_3'+m)$.
Let $X_*(T)_+'$ be $X_*(T)_+$ modulo $\sim$, and $[\lambda]=[m_1, m_2, m_3]$ denotes the class of a dominant cocharacter $\lambda=(m_1, m_2, m_3)$.
For any $[\lambda],[\mu]\in X_*(T)_+'$, we write $[\lambda]\preceq [\mu]$ if there exists $M=(m,m,m)$ such that $\lambda+M\preceq \mu$.

For any vertex $P$ in the Bruhat-Tits building of $\mathrm{SL}_3(L)$, we can choose a matrix $x\in \mathrm{GL}_3(L)$ such that $P=[x\Lambda_{\bar k}]$, where $[\mathscr L]$ denotes the class of a lattice $\mathscr L$.
Then, for any two vertices $P=[x\Lambda_{\bar k}], Q=[y\Lambda_{\bar k}]\in \mathcal B_{\infty}$, one can define their relative position by (the class of) $\mathrm{inv}(xK,yK)\in X_*(T)_+'$, which is clearly independent of the choice of $x$ and $y$.
We denote it by $\mathrm{inv}'(P,Q)$.

Let $S$ be a subset of $\mathcal B_{\infty}$ contained in some apartment, and define $cl(S)$ to be the intersection of all the apartments which contain $S$.
If $S$ is contained in an apartment $\mathcal A$ of $\mathcal B_{\infty}$, then $cl(S)$ is equal to the intersection of all half-apartments in $\mathcal A$ containing $S$.
Let $S_1$ and $S_2$ be two faces in $\mathcal B_{\infty}$.
We define $cl(S_1, S_2)$ to be $cl(S_1\cup S_2)$.
If each of $S_1$ and $S_2$ consists of a single vertex of $\mathcal B_{\infty}$, then we will write $cl(P_1,P_2)$ for it, where $S_1=\{P_1\}$ and $S_2=\{P_2\}$.
Let $[e_1, e_2, e_3]=\mathrm{inv}'(P_2,P_1)$, and let $\mathcal A$ be an apartment containing $P_1$ and $P_2$.
If $e_1=e_2$ or $e_2=e_3$, then there is a wall in $\mathcal A$ which contains $P_1$ and $P_2$, and $cl(P_1,P_2)$ consists of the vertices of the wall which lie between $P_1$ and $P_2$.
\begin{center}
\begin{tikzpicture}
 \draw (-2,0)--(0,0)--(2,0)--(4,0)--cycle;
 \draw (-2,0)node[left]{$P_1$};
 \draw (4,0)node[right]{$P_2$};
 \draw (-2,0)node{$\bullet$};
 \draw (0,0)node{$\bullet$};
 \draw (2,0)node{$\bullet$};
 \draw (4,0)node{$\bullet$};
\end{tikzpicture}
\end{center}
In this case we will say that $cl(P_1,P_2)$ is a line segment.
Note that each minimal gallery $(C_0,\ldots,C_r)$ connecting $P_1$ and $P_2$ is of the form
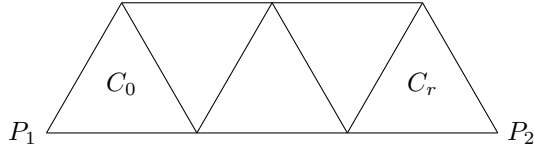
\begin{center}
\begin{tikzpicture}
 \draw (-2,0)--(0,0)--(2,0)--(4,0)--cycle;
 \draw (-1,1.732)--(1,1.732)--(3,1.732)--cycle;
 \draw (-2,0)--(-1,1.732)--cycle;
 \draw (0,0)--(-1,1.732)--cycle;
 \draw (0,0)--(1,1.732)--cycle;
 \draw (2,0)--(1,1.732)--cycle;
 \draw (2,0)--(3,1.732)--cycle;
 \draw (4,0)--(3,1.732)--cycle;
 \draw (-2,0)node[left]{$P_1$};
 \draw (4,0)node[right]{$P_2$};
 \draw (-1,0.65)node{$C_0$};
 \draw (3,0.65)node{$C_r$};
\end{tikzpicture}
\end{center}
where the points in the bottom row are the points of $cl(P_1,P_2)$.
Such a gallery is called a {\it special gallery} connecting $P_1$ and $P_2$.
We call $C_0$ a first chamber of $cl(P_1,P_2)$ ($C_0$ is a chamber containing $P_1$).
If $e_1>e_2>e_3$, then $cl(P_1, P_2)$ forms a parallelogram.
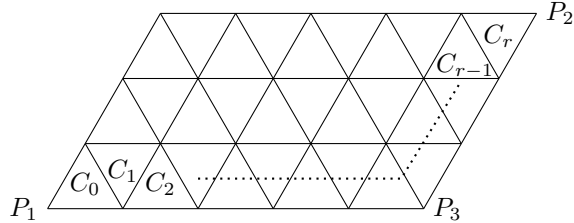
\begin{center}
\begin{tikzpicture}
 \draw (-3,0)--(-2,0)--(-1,0)--(0,0)--(1,0)--(2,0)--cycle;
 \draw (-2.5,0.866)--(-1.5,0.866)--(-0.5,0.866)--(0.5,0.866)--(1.5,0.866)--(2.5,0.866)--cycle;
 \draw (-3,0)--(-2.5,0.866)--cycle;
 \draw (-2,0)--(-2.5,0.866)--cycle;
 \draw (-2,0)--(-1.5,0.866)--cycle;
 \draw (-1,0)--(-1.5,0.866)--cycle;
 \draw (-1,0)--(-0.5,0.866)--cycle;
 \draw (0,0)--(-0.5,0.866)--cycle;
 \draw (0,0)--(0.5,0.866)--cycle;
 \draw (1,0)--(0.5,0.866)--cycle;
 \draw (1,0)--(1.5,0.866)--cycle;
 \draw (2,0)--(1.5,0.866)--cycle;
 \draw (2,0)--(2.5,0.866)--cycle;
 \draw (-2,1.732)--(-1,1.732)--(0,1.732)--(1,1.732)--(2,1.732)--(3,1.732)--cycle;
 \draw (-2,1.732)--(-2.5,0.866)--cycle;
 \draw (-2,1.732)--(-1.5,0.866)--cycle;
 \draw (-1,1.732)--(-1.5,0.866)--cycle;
 \draw (-1,1.732)--(-0.5,0.866)--cycle;
 \draw (0,1.732)--(-0.5,0.866)--cycle;
 \draw (0,1.732)--(0.5,0.866)--cycle;
 \draw (1,1.732)--(0.5,0.866)--cycle;
 \draw (1,1.732)--(1.5,0.866)--cycle;
 \draw (2,1.732)--(1.5,0.866)--cycle;
 \draw (2,1.732)--(2.5,0.866)--cycle;
 \draw (3,1.732)--(2.5,0.866)--cycle;
 \draw (-1.5,2.598)--(-0.5,2.598)--(1.5,2.598)--(2.5,2.598)--(3.5,2.598)--cycle;
 \draw (-2,1.732)--(-1.5,2.598)--cycle;
 \draw (-1,1.732)--(-1.5,2.598)--cycle;
 \draw (-1,1.732)--(-0.5,2.598)--cycle;
 \draw (0,1.732)--(-0.5,2.598)--cycle;
 \draw (0,1.732)--(0.5,2.598)--cycle;
 \draw (1,1.732)--(0.5,2.598)--cycle;
 \draw (1,1.732)--(1.5,2.598)--cycle;
 \draw (2,1.732)--(1.5,2.598)--cycle;
 \draw (2,1.732)--(2.5,2.598)--cycle;
 \draw (3,1.732)--(2.5,2.598)--cycle;
 \draw (3,1.732)--(3.5,2.598)--cycle;
 \draw (-3,0)node[left]{$P_1$};
 \draw (3.5,2.598)node[right]{$P_2$};
 \draw (2,0)node[right]{$P_3$};
 \draw (-2.5,0.325)node{$C_0$};
 \draw (-2,0.5)node{$C_1$};
 \draw (-1.5,0.325)node{$C_2$};
 \draw (2.55,1.9)node{$C_{r-1}$};
 \draw (3,2.3)node{$C_r$};
 \draw[dotted, thick] (-1,0.4)--(1.7,0.4)--(2.5,1.7);
\end{tikzpicture}
\end{center}
In this case, the gallery $(C_0,\ldots,C_r)$ right above is called the special gallery connecting $P_1$ and $P_2$, where $\mathrm{inv}'(P_3,P_1)=[e_2-e_3,e_2-e_3,0]$.
We call $C_0$ the first chamber of $cl(P_1,P_2)$ ($C_0$ is the chamber containing $P_1$).

Let $P_1$ be a vertex of $\mathcal B_{\infty}$.
Let $C_1,C_2$  be two chambers containing $P_1$, and let $P_2$ be a vertex of $C_1$ such that $\mathrm{inv}'(P_2,P_1)=[1,0,0]$.
Then the following three figures illustrates three cases of the relative position of $C_1$ and $C_2$.
We say that $C_1$ and $C_2$ have {\it relative position I} (resp.\ {\it relative position I\hspace{-.1em}I}, resp.\ {\it relative position I\hspace{-.1em}I\hspace{-.1em}I}) if $cl(C_1,C_2)$ looks like Figure \ref{relative position 1} (resp.\ Figure \ref{relative position 2}, resp.\ Figure \ref{relative position 3}).

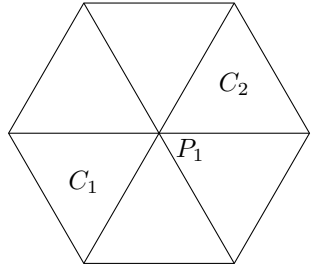
\begin{figure}[h]
\centering
\begin{tikzpicture}
 \draw (-2,0)--(-1,1.732)--(1,1.732)--(2,0)--(1,-1.732)--(-1,-1.732)--cycle;
 \draw (0,0)--(-2,0)--cycle;
 \draw (0,0)--(-1,1.732)--cycle;
 \draw (0,0)--(1,1.732)--cycle;
 \draw (0,0)--(2,0)--cycle;
 \draw (0,0)--(1,-1.732)--cycle;
 \draw (0,0)--(-1,-1.732)--cycle;
 \draw (0.1,0.05)node[below right]{$P_1$};
 \draw (1,0.65)node{$C_2$};
 \draw (-1,-0.65)node{$C_1$};
\end{tikzpicture}
\caption{Relative position I}
\label{relative position 1}
\end{figure}
\begin{figure}[h]
\begin{minipage}{0.5\columnwidth}
\centering
\begin{tikzpicture}
 \draw (-2,0)--(2,0)--(1,-1.732)--(-1,-1.732)--cycle;
 \draw (0,0)--(-2,0)--cycle;
 \draw (0,0)--(2,0)--cycle;
 \draw (0,0)--(1,-1.732)--cycle;
 \draw (0,0)--(-1,-1.732)--cycle;
 \draw (0,0)node[above]{$P_1$};
 \draw (-2,0)node[above]{$P_2$};
 \draw (-1,-0.65)node{$C_1$};
 \draw (1,-0.65)node{$C_2$};
\end{tikzpicture}
\caption{Relative position I\hspace{-.1em}I}
\label{relative position 2}
\end{minipage}
\begin{minipage}{0.5\columnwidth}
\centering
\begin{tikzpicture}
 \draw (-2,0)--(2,0)--(1,-1.732)--(-1,-1.732)--cycle;
 \draw (0,0)--(-2,0)--cycle;
 \draw (0,0)--(2,0)--cycle;
 \draw (0,0)--(1,-1.732)--cycle;
 \draw (0,0)--(-1,-1.732)--cycle;
 \draw (0,0)node[above]{$P_1$};
 \draw (-0.9,-1.8)node[above right]{$P_2$};
 \draw (-1,-0.65)node{$C_1$};
 \draw (1,-0.65)node{$C_2$};
\end{tikzpicture}
\caption{Relative position I\hspace{-.1em}I\hspace{-.1em}I}
\label{relative position 3}
\end{minipage}
\end{figure}

Let $\mathcal A$ be an apartment of $\mathcal B_{\infty}$ and $C$ a chamber contained in $\mathcal A$.
Then we denote by $\rho_{\mathcal A, C}$ the retraction of $\mathcal B_{\infty}$ onto $\mathcal A$ with center $C$.
From now on, we will consider the following situation:
Let $P_0,P_1,P_2$ belong to $\mathcal B_{\infty}$ with $P_1\neq P_0$ and $P_2\neq P_0$.
Let $(C_0,\ldots,C_r)$ (resp.\ $(D_0,\ldots,D_s)$) be a special gallery connecting $P_0$ to $P_1$ (resp.\ $P_0$ to $P_2$).
Moreover, we assume that $C_0$ and $D_0$ have relative position I, I\hspace{-.1em}I or I\hspace{-.1em}I\hspace{-.1em}I.
We want to compute $\mathrm{inv}'(P_1, P_2)$ from $\mathrm{inv}'(P_0, P_1)$ and $\mathrm{inv}'(P_0, P_2)$.
For this, the main machineries are the following two lemmas.

\begin{lemm}
\label{LEMMA1.1}
Let $C$ be a chamber of $\mathcal B_{\infty}$, and let $\mathcal A$ be an apartment containing $C$.
If a vertex $P$ belongs to $C$ and a vertex $Q$ belongs to $\mathcal B_{\infty}$, then $\mathrm{inv}'(P, Q)=\mathrm{inv}'(P, \rho (Q))$, where $\rho=\rho_{\mathcal A,C}$.
\end{lemm}
\begin{proof}
This is \cite[LEMMA 1.1]{Kottwitz}.
\end{proof}

\begin{lemm}
\label{LEMMA1.2}
Let $\mathcal A$ be an apartment of $\mathcal B_{\infty}$, $C$ a chamber of $\mathcal A$ and $\rho=\rho_{\mathcal A,C}$.
Let $E$ be an edge of $\mathcal B_{\infty}$, and let $C_1$ (resp.\ $C_2$) be the chamber of $\mathcal A$ containing $\rho(E)$ which is on the same (resp.\ opposite) side as $C$ of the wall of $\mathcal A$ containing $\rho(E)$.
Then there is a unique chamber $C'$ containing $E$ such that $\rho(C')=C_1$, and for every chamber $C''$ containing $E$ which is distinct from $C'$, we have $\rho(C'')=C_2$.
\end{lemm}
\begin{proof}
This is \cite[LEMMA 1.2]{Kottwitz}.
\end{proof}

Let $\mathcal A$ be an apartment containing $C_0,\ldots, C_r$, and let $\rho=\rho_{\mathcal A, C_r}$ be the retraction of $\mathcal B_{\infty}$ onto $\mathcal A$ with center $C_r$.
Our method will be to retract the gallery $(D_0,\ldots, D_s)$ onto the apartment $\mathcal A$, using Lemma \ref{LEMMA1.2} each step of the way.
Lemma \ref{LEMMA1.1} guarantees that such retraction does not change $\mathrm{inv}'(P_1, P_2)$.

Let $E_i$ be the edge between $D_i$ and $D_{i+1}$.
If $\rho(D_i)$ is on the same side of $\rho(E_i)$ as $C_r$, then $\rho(D_{i+1})$ is the other chamber in $\mathcal A$ besides $\rho(D_i)$ containing $\rho(E_i)$.
However, if $\rho(D_i)$ is on the opposite side of $\rho(E_i)$ as $C_r$, then it may happen that $E_{i+1}$ retracts to $\rho(E_i)$.
In this case, we will say that ``a bend has occurred'' at the edge between $D_i$ and $D_{i+1}$.
There is at most one edge at which a bend occurs (\cite[p. 342]{Kottwitz}).
If $C_0$ and $D_0$ have relative position I, then a bend never occurs ($\rho(D_i)$ is always on the same side of $\rho(E_i)$ as $C_r$).
If $C_0$ and $D_0$ have relative position I\hspace{-.1em}I, then $\rho(D_i)$ is on the same side of $\rho(E_i)$ as $C_r$ for $i=-1, 0,\ldots, a$, as shown in Figure \ref{bend2}.
However $\rho(D_{a+2j-1})$ ($1\le j,a+2j-1\le s$) is on the side of $\rho(E_{a+2j-1})$ opposite from $C_r$, so a bend can occur at the edge $E_{a+2j-1}$.
One can show that the condition that a bend occur at the edge between $D_{a+2j_0}$ and $D_{a+2j_{0}-1}$ is
$$\text{$D_{a+2j}\subset cl(D_{a+2j-1},C_{2j})$ for $1\le j<j_0$, and 
$D_{a+2j_0}\nsubseteq cl(D_{a+2j_0-1},C_{2j_0})$}$$
or, equivalently,
$$\text{$C_{2j}\subset cl(D_{a+2j},C_{2j-1})$ for $1\le j<j_0$, and 
$C_{2j_0}\nsubseteq cl(D_{a+2j_0},C_{2j_0-1})$}.$$
If $C_0$ and $D_0$ have relative position I\hspace{-.1em}I\hspace{-.1em}I, then the condition that a bend occur at the edge between $D_{2j_0}$ and $D_{2j_{0}-1}$ is
$$\text{$D_{2j}\subset cl(D_{2j-1},C_{a+2j})$ for $1\le j<j_0$, and 
$D_{2j_0}\nsubseteq cl(D_{2j_0-1},C_{a+2j_0})$}$$
or, equivalently,
$$\text{$C_{a+2j}\subset cl(D_{2j},C_{a+2j-1})$ for $1\le j<j_0$, and 
$C_{a+2j_0}\nsubseteq cl(D_{2j_0},C_{a+2j_0-1})$}.$$
For more details on these conditions, see Case I\hspace{-.1em}I and Case I\hspace{-.1em}I\hspace{-.1em}I in \cite[Section 2]{Kottwitz}.
\begin{figure}[h]
\begin{center}
\begin{tikzpicture}
 \draw (2.5,-0.866)--(3.5,-0.866)--cycle;
 \draw (-4,0)--(-3,0)--(-2,0)--cycle;
 \draw (2,0)--(3,0)--cycle;
 \draw (-3.5,0.866)--(-2.5,0.866)--(-1.5,0.866)--cycle;
 \draw (1.5,0.866)--(2.5,0.866)--cycle;
 \draw (2.5, -0.866)--(2,0)--cycle;
 \draw (3.5, -0.866)--(3,0)--cycle;
 \draw (2.5, -0.866)--(3,0)--cycle;
 \draw (-4,0)--(-3.5,0.866)--cycle;
 \draw (-3,0)--(-3.5,0.866)--cycle;
 \draw (-3,0)--(-2.5,0.866)--cycle;
 \draw (-2,0)--(-2.5,0.866)--cycle;
 \draw (-2,0)--(-1.5,0.866)--cycle;
 \draw (2,0)--(1.5,0.866)--cycle;
 \draw (2,0)--(2.5,0.866)--cycle;
 \draw (3,0)--(2.5,0.866)--cycle;
 \draw (-2,1.732)--(-1,1.732)--(0,1.732)--(1,1.732)--(2,1.732)--cycle;
 \draw (-2,1.732)--(-2.5,0.866)--cycle;
 \draw (-2,1.732)--(-1.5,0.866)--cycle;
 \draw (-1,1.732)--(-1.5,0.866)--cycle;
 \draw (1,1.732)--(1.5,0.866)--cycle;
 \draw (2,1.732)--(1.5,0.866)--cycle;
 \draw (2,1.732)--(2.5,0.866)--cycle;
 \draw (-1.5,2.598)--(-0.5,2.598)--(1.5,2.598)--cycle;
 \draw (-2,1.732)--(-1.5,2.598)--cycle;
 \draw (-1,1.732)--(-1.5,2.598)--cycle;
 \draw (-1,1.732)--(-0.5,2.598)--cycle;
 \draw (0,1.732)--(-0.5,2.598)--cycle;
 \draw (0,1.732)--(0.5,2.598)--cycle;
 \draw (1,1.732)--(0.5,2.598)--cycle;
 \draw (1,1.732)--(1.5,2.598)--cycle;
 \draw (2,1.732)--(1.5,2.598)--cycle;
 \draw (-3.5,0.325)node{$C_r$};
 \draw (-1,2.3)node{$C_0$};
 \draw (-0.5,1.9)node{$D_{-1}$};
 \draw (0,2.3)node{$D_0$};
 \draw (1,2.3)node{$D_a$};
 \draw (1.54,1.9)node{$D_{a+1}$};
 \draw (3,-0.6)node{$D_s$};
\end{tikzpicture}
\end{center}
\caption{Computation for relative position I\hspace{-.1em}I}
\label{bend2}
\end{figure}

\begin{figure}[h]
\begin{center}
\begin{tikzpicture}
 \draw (-1,0)--(0,0)--(1,0)--(2,0)--cycle;
 \draw (-1.5,0.866)--(-0.5,0.866)--(0.5,0.866)--(1.5,0.866)--(2.5,0.866)--cycle;
 \draw (-1,0)--(-1.5,0.866)--cycle;
 \draw (-1,0)--(-0.5,0.866)--cycle;
 \draw (0,0)--(-0.5,0.866)--cycle;
 \draw (0,0)--(0.5,0.866)--cycle;
 \draw (1,0)--(0.5,0.866)--cycle;
 \draw (1,0)--(1.5,0.866)--cycle;
 \draw (2,0)--(1.5,0.866)--cycle;
 \draw (2,0)--(2.5,0.866)--cycle;
 \draw (-2,1.732)--(-1,1.732)--cycle;
 \draw (2,1.732)--(3,1.732)--(4,1.732)--cycle;
 \draw (-2,1.732)--(-1.5,0.866)--cycle;
 \draw (-1,1.732)--(-1.5,0.866)--cycle;
 \draw (-1,1.732)--(-0.5,0.866)--cycle;
 \draw (2,1.732)--(1.5,0.866)--cycle;
 \draw (2,1.732)--(2.5,0.866)--cycle;
 \draw (3,1.732)--(2.5,0.866)--cycle;
 \draw (-2.5,2.598)--(-1.5,2.598)--cycle;
 \draw (2.5,2.598)--(3.5,2.598)--(4.5,2.598)--cycle;
 \draw (-2,1.732)--(-2.5,2.598)--cycle;
 \draw (-2,1.732)--(-1.5,2.598)--cycle;
 \draw (-1,1.732)--(-1.5,2.598)--cycle;
 \draw (2,1.732)--(2.5,2.598)--cycle;
 \draw (3,1.732)--(2.5,2.598)--cycle;
 \draw (3,1.732)--(3.5,2.598)--cycle;
 \draw (4,1.732)--(3.5,2.598)--cycle;
 \draw (4,1.732)--(4.5,2.598)--cycle;
 \draw (-3,3.464)--(-2,3.464)--cycle;
 \draw (-3,3.464)--(-2.5,2.598)--cycle;
 \draw (-2,3.464)--(-2.5,2.598)--cycle;
 \draw (-2,3.464)--(-1.5,2.598)--cycle;
 \draw (0.5,0.325)node{$C_0$};
 \draw (-0.5,0.325)node{$C_a$};
 \draw (-2.5,3.15)node{$C_r$};
 \draw (1,0.67)node{$D_{-1}$};
 \draw (1.5,0.325)node{$D_0$};
 \draw (4,2.3)node{$D_s$};
\end{tikzpicture}
\end{center}
\caption{Computation for relative position I\hspace{-.1em}I\hspace{-.1em}I}
\label{bend3}
\end{figure}
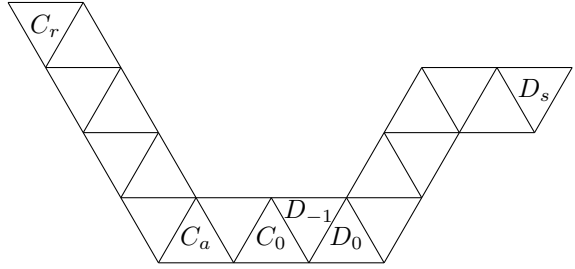

The retraction of $(D_0,\ldots,D_s)$ onto $\mathcal A$ allows us to compute $\mathrm{inv}'(P_1,P_2)$.
The next proposition is the summary of this computation (cf.\  \cite[LEMMA 2.1]{Kottwitz}).
\begin{prop}
\label{invformula}
Let $P_0,P_1,P_2$ belong to $\mathcal B_{\infty}$ with $P_1\neq P_0$ and $P_2\neq P_0$.
Let $[e_1,e_2,e_3]=\mathrm{inv}'(P_0,P_1)$ and let $[f_1,f_2,f_3]=\mathrm{inv}'(P_0,P_2)$.
Let $(C_0,\ldots,C_r)$ (resp.\ $(D_0,\ldots,D_s)$) be a special gallery connecting $P_0$ to $P_1$ (resp.\ $P_0$ to $P_2$).
Assume that $C_0$ and $D_0$ have relative position I, I\hspace{-.1em}I or I\hspace{-.1em}I\hspace{-.1em}I.
\begin{enumerate}[(i)]
\item If $C_0$ and $D_0$ have relative position I, then $$\mathrm{inv}'(P_1,P_2)=[f_1-e_3,f_2-e_2,f_3-e_1].$$
\item Let $m=\min\{e_1-e_2,f_2-f_3\}$.
If $C_0$ and $D_0$ have relative position I\hspace{-.1em}I, then $\mathrm{inv}'(P_1,P_2)=[f_1-e_3,f_2-e_2-j,f_3-e_1+j]$, where $j=m$ if no bend occurs, and $j=j_0$ if a bend occurs at the edge between $D_{a+2j_0-1}$ and $D_{a+2j_0}$ (Figure \ref{bend2}).
The possible values for $j_0$ are $1,2,\ldots, m-1$.
\item Let $m=\min\{f_1-f_2,e_2-e_3\}$.
If $C_0$ and $D_0$ have relative position I\hspace{-.1em}I\hspace{-.1em}I, then $\mathrm{inv}'(P_1,P_2)=[f_1-e_3-j,f_2-e_2+j,f_3-e_1]$, where $j=m$ if no bend occurs, and $j=j_0$ if a bend occurs at the edge between $D_{2j_0-1}$ and $D_{2j_0}$ (Figure \ref{bend3}).
The possible values for $j_0$ are $1,2,\ldots, m-1$.
\end{enumerate}
\end{prop}

\begin{rema}
In the cases (ii) and (iii) of Proposition \ref{invformula}, it may happen that $m=0$.
If it occurs, then $cl(P_0, P_1)$ or $cl(P_0, P_2)$ is a line segment and we can take a pair $(C_0, D_0)$ such that $C_0$ and $D_0$ have relative position I.
However, this does not contradict the results in Proposition \ref{invformula} because $m=0$ implies that $j$ is always equal to $0$.
\end{rema}

\section{Affine Deligne-Lusztig Sets}
\label{ADLVset}
In the following, we will consider the set $X_{\lambda}(b)(\bar k)\cap \eta^{-1}(0)$.
Note that we have
$$X_{\lambda}(b)(\bar k)\cap \eta^{-1}(0)=\{xK^S\in \mathrm{SL}_3(L)/K^S\mid x^{-1}b\sigma(x)\in Kt^{\lambda}K\}.$$
In terms of buildings, this is equal to the set
$$\{P\mid \text{$P$ is a vertex of type $0$ in $\mathcal B_{\infty}$ with $\mathrm{inv}'(P,b\sigma P)=[\lambda]$}\}.$$

\subsection{Affine Deligne-Lusztig Varieties for $\mathrm{SL}_3$}
\label{b=1}
In this subsection, We will study the affine Deligne-Lusztig varieties for $\mathrm{SL}_3$ and $b=1$.

Let $Q$ be a vertex in $\mathcal B_1$, and let $[\mu]=[e_1,e_2,e_3]\in X_*(T)_+'$.
We denote by $\mathcal{G}r_{[\mu]}(Q)$ the non-empty subset
$$\{P\mid\text{$P$ is a vertex in $\mathcal B_{\infty}$ with $\mathrm{inv}'(Q,P)=[\mu]$}\}.$$
Clearly, all vertices in $\mathcal{G}r_{[\mu]}(Q)$ have the same type.
Let $i\in \{0,1,2\}$ be the type of vertices in this set.
Then, by choosing the $i$-special representative, we always see $\mathcal{G}r_{[\mu]}(Q)$ as a subset of $\mathcal{L}att^i(\bar k)$.
Define
\begin{equation*}
G_{[\mu]}^{\mathcal B_1}(Q)=
\left \{P\in \mathcal{G}r_{[\mu]}(Q) \left|
\begin{array}{l}
\text{there exists a minimal gallery from $Q$ to $P$}\\
\text{containing no vertices in $\mathcal B_1$ except $Q$}
\end{array}
\right.\right\}.
\end{equation*}
Let $\lambda=(m_1,m_2,m_3)\in X_*(T^S)_+$ such that $X_{\lambda}^S(1)$ is non-empty (for the explicit criterion, see Remark \ref{nonemptiness}).
Set
$$M_{\lambda}(1)= \{[\mu]\in X_*(T)_+'\mid \text{$X_{\lambda}^S(1)(\bar k)\cap G_{[\mu]}^{\mathcal B_1}(Q)\neq \emptyset$ for some vertex $Q$ in $\mathcal B_1$}\}.$$
 
\begin{lemm}
\label{invfinite}
Let notation be as above.
We have
\[
  M_{\lambda}(1)=
  \begin{cases}
    \mathbb \{[e_1,e_2,e_3]\in X_*(T)_+'\mid e_1-e_3=m_1\} & (m_2=0) \\
    \mathbb \{[e_1,e_2,e_3]\in X_*(T)_+'\mid e_1-e_3=m_1, m_{\mu}\geq -m_2\} & (m_2<0) \\
    \mathbb \{[e_1,e_2,e_3]\in X_*(T)_+'\mid e_3-e_1=m_3, m_{\mu}\geq m_2\} & (m_2>0), \\
  \end{cases}
\]
where $m_{\mu}=\min\{e_1-e_2,e_2-e_3\}$.
In particular, $M_{\lambda}(1)$ is a finite set.
\end{lemm}
\begin{proof}
Let $P\in G_{[\mu]}^{\mathcal B_1}(Q)$ and let $(C_0, \ldots, C_r)$ be a special gallery connecting $Q$ to $P$,
where $[\mu]=[e_1,e_2,e_3]\in M_{\lambda}(1)$.
Then it follows that $\lambda$ is one of the forms
\begin{enumerate}[(i)]
\item $(e_1-e_3,0,e_3-e_1)$
\item $(e_1-e_3,-j,e_3-e_1+j),\quad \min\{1,m_{\mu}\}\le j\le m_{\mu}$
\item $(e_1-e_3-j,j,e_3-e_1),\quad \min\{1,m_{\mu}\}\le j\le m_{\mu}$,
\end{enumerate}
where $m_{\mu}=\min\{e_1-e_2,e_2-e_3\}$ (and note that $m_{\mu}=0$ implies $j=0$).
Indeed, using Proposition \ref{invformula}, we can compute $\mathrm{inv}'(P, \sigma P)$ by connecting $(C_0, \ldots, C_r)$ and $(\sigma C_0, \ldots, \sigma C_r)$.
Then, by Lemma \ref{ADLVbasic}, $\lambda=(m_1, m_2, m_3)$ is the representative of $\mathrm{inv}'(P, \sigma P)$ with $m_1+m_2+m_3=0$.

Conversely, for any $[\mu]=[e_1,e_2,e_3]\in X_*(T)_+'$ with $e_1>e_2>e_3$ and $\lambda=(e_1-e_3,-j_0,e_3-e_1+j_0)\in X_*(T^S)_+$ with $1\le j_0< m_{\mu}$ (resp.\ $\lambda=(e_1-e_3,-m_{\mu},e_3-e_1+m_{\mu})\in X_*(T^S)_+$), there exists $P\in G_{[\mu]}^{\mathcal B_1}(Q)$ belonging to $X_{\lambda}^S(1)(\bar k)$.
If $m_{\mu}=1$, then this is obvious.
So we may assume $m_{\mu}\geq 2$.
To check the assertion, we use the condition that a bend occur (Section \ref{building}).
Let $a=2(e_1-e_2-1)$.
In the case (ii), $\lambda$ is of the form $(e_1-e_3,-j_0,e_3-e_1+j_0)$ (resp.\ $\lambda=(e_1-e_3,-m_{\mu},e_3-e_1+m_{\mu})$) if and only if
\begin{align*}
&\text{$\sigma C_{a+2j}\subset cl(\sigma C_{a+2j-1},C_{2j})$ for $1\le j<j_0$, and}\\
&\sigma C_{a+2j_0}\nsubseteq cl(\sigma C_{a+2j_0-1},C_{2j_0})\\
&\text{(resp.\ $\sigma C_{a+2i}\subset cl(\sigma C_{a+2i-1},C_{2i})$ for $1\le i<m_{\mu}$).}
\end{align*}
So, for any $j$ with $1\le j<j_0$, $\sigma C_{a+2j}$ (hence $C_{a+2j}$) is the unique chamber determined by the gallery $(C_0, \ldots, C_{2j}, \ldots, C_{a+2j-1})$, and $\sigma C_{a+2j_0}$ (hence $C_{a+2j_0}$) is a chamber distinct from the one uniquely determined by the gallery $(C_0, \ldots, C_{2j_0}, \ldots, C_{a+2j_0-1})$ (such $C_{a+2j_0}$ exists because the Bruhat-Tits building of $\mathrm{SL}_3(L)$ is actually a {\it thick} building).
Similarly, for any $i$ with $1\le i<m_{\mu}$, $\sigma C_{a+2i}$ (hence $C_{a+2i}$) is the unique chamber determined by the gallery $(C_0, \ldots, C_{2i}, \ldots, C_{a+2i-1})$.
This proves the existence of $P\in G_{[\mu]}^{\mathcal B_1}(Q)$ belonging to $X_{\lambda}^S(1)$ for any $\lambda$ as in (ii) above, and the same is true for any $\lambda$ as in (iii).

In any case, the positive value $e_1-e_3$ is determined by $\lambda$, and we can always choose the representative of $[\mu]$ with $e_2=0, e_1\geq 0\geq e_3$.
Clearly, the set of tuples $(e_1,e_3)$ satisfying these conditions is finite, and thus the lemma follows.
\end{proof}

\begin{lemm}
\label{invclosed}
Let $Q$ be a vertex in $\mathcal B_1$.
If $X_{\lambda}^S(1)(\bar k)\cap G_{[\mu]}^{\mathcal B_1}(Q)$ is non-empty, then we have $$X_{\lambda}^S(1)(\bar k)\cap G_{[\mu]}^{\mathcal B_1}(Q)=X_{\lambda}^S(1)(\bar k)\cap \mathcal{G}r_{[\mu]}(Q)=X_{\lambda}^S(1)(\bar k)\cap \overline{\mathcal{G}r_{[\mu]}(Q)},$$
where $\overline{\mathcal{G}r_{[\mu]}(Q)}$ denotes the closure of $\mathcal{G}r_{[\mu]}(Q)$ in $X^S$.
In particular, $X_{\lambda}^S(1)(\bar k)\cap G_{[\mu]}^{\mathcal B_1}(Q)$ is closed in $X_{\lambda}(1)(\bar k)$.
\end{lemm}
\begin{proof}
It is enough to show $X_{\lambda}^S(1)(\bar k)\cap G_{[\mu]}^{\mathcal B_1}(Q)=X_{\lambda}^S(1)(\bar k)\cap \overline{\mathcal{G}r_{[\mu]}(Q)}$.
Note that we have $$\overline{\mathcal{G}r_{[\mu]}(Q)}=\bigcup_{[\mu']\preceq [\mu]}\mathcal{G}r_{[\mu']}(Q).$$
Set $\mu=(e_1,e_2,e_3), \mu'=(e_1',e_2',e_3')$ with $[\mu']\preceq [\mu]$.
Let $\mu''=(e_1'',e_2'',e_3'')$ be a dominant cocharacter satisfying $\mu'-\mu''\in X_*(T)_+$.
Then it suffices to show that $X_{\lambda}^S(1)(\bar k)\cap G^{\mathcal B_1}_{[\mu'']}(Q)$ is the empty set unless $[\mu'']=[\mu]$.
Indeed, for any $P\in \overline{\mathcal{G}r_{[\mu]}(Q)}$, there exists such $\mu''$ with $P\in G^{\mathcal B_1}_{[\mu'']}(Q)$.

We may assume that $\mu'\preceq \mu$.
We have $e_1''-e_3''\le e_1'-e_3'\le e_1-e_3$.
The formula in the proof of Lemma \ref{invfinite} shows that if $P\in G^{\mathcal B_1}_{[\mu'']}(Q)$ is contained in $X_{\lambda}^S(1)(\bar k)$, then $e_1''-e_3''=e_1-e_3$, and this equation implies $[\mu'']=[\mu']=[\mu]$.
So $X_{\lambda}^S(1)(\bar k)\cap G^{\mathcal B_1}_{[\mu'']}(Q)\neq \emptyset$ if and only if $[\mu'']=[\mu]$.
The last assertion follows from Proposition \ref{ADLVdecomposition}.
\end{proof}

For any $\lambda\in X_*(T^S)_+$ such that $X_{\lambda}^S(1)$ is non-empty, set
$$\mathcal P_{\lambda}(1)=\{(Q,[\mu]) \mid \text{$Q$ is a vertex in $\mathcal B_1$ with $X_{\lambda}^S(1)(\bar k)\cap G_{[\mu]}^{\mathcal B_1}(Q)\neq \emptyset$}\}.$$
Then for any $[\mu]=[e_1,e_2,e_3]\in M_{\lambda}(1)$, a tuple $(Q,[\mu])$ is contained in $\mathcal P_{\lambda}(1)$ if and only if $Q$ is a vertex of type $-(e_1+e_2+e_3)\in \mathbb Z/3$  in $\mathcal B_1$.
Indeed, the action of $\mathrm{SL}_3(F)$ on all vertices with the same type in $\mathcal B_1$ is transitive.

\begin{prop}
\label{invdecomposition}
For any $\lambda\in X_*(T^S)_+$ such that $X_{\lambda}^S(1)$ is non-empty, we have a decomposition 
$$X_{\lambda}^S(1)(\bar k)=\bigcup_{(Q,[\mu])\in\mathcal P_{\lambda}(1)}(X_{\lambda}^S(1)(\bar k)\cap G_{[\mu]}^{\mathcal B_1}(Q)).$$
\end{prop}

\begin{proof}
For any $P\in X_{\lambda}^S(1)(\bar k)$, we have a minimal gallery to $\mathcal B_1$.
Let $Q$ be a vertex in $\mathcal B_1$ such that the distance between $P$ and $Q$ is minimal.
Then $P$ is contained in $X_{\lambda}^S(1)(\bar k)\cap G_{[\mu]}^{\mathcal B_1}(Q)$, where $[\mu]=\mathrm{inv}'(Q,P)$.
So we obtain the decomposition in the proposition.
\end{proof}

\subsection{The Superbasic Case}
\label{superbasic}
We keep the notation above.
Set $$
b_1=\begin{pmatrix}
0 & 0 & t \\
1 & 0 & 0 \\
0 & 1 & 0 \\
\end{pmatrix},\quad
b_2=\begin{pmatrix}
0 & t & 0 \\
0 & 0 & t \\
1 & 0 & 0 \\
\end{pmatrix}.
$$
Then the newton vector of $b_1$ (resp.\ $b_2$) is $(\frac{1}{3},\frac{1}{3},\frac{1}{3})$ (resp.\ $(\frac{2}{3},\frac{2}{3},\frac{2}{3})$).

Let $C_M$ be the main chamber consisting of three vertices $[\mathcal O\oplus \mathcal O\oplus \mathcal O], [t\mathcal O\oplus \mathcal O\oplus \mathcal O], [t\mathcal O\oplus t\mathcal O\oplus \mathcal O]$, and let $Q$ be a vertex in $C_M$.
Let $\lambda=(m_1,m_2,m_3)\in X_*(T)_+$ with $X_{\lambda}(b_i)\neq \emptyset$ $(i=1,2)$ (for the explicit criterion, see Remark \ref{nonemptiness}).
Set
\begin{equation*}
M_{\lambda}(b_i)=
\left \{[\mu]\in X_*(T)_+' \left|
\begin{array}{l}
\text{$(X_{\lambda}(b_i)(\bar k)\cap \eta^{-1}(0))\cap G_{[\mu]}^{C_M}(Q)\neq \emptyset$}\\
\text{for some vertex $Q$ in $C_M$}
\end{array}
\right.\right\}
\end{equation*}
where
\begin{equation*}
G_{[\mu]}^{C_M}(Q)=
\left \{P\in \mathcal{G}r_{[\mu]}(Q) \left|
\begin{array}{l}
\text{there exists a minimal gallery from $Q$ to $P$}\\
\text{containing no vertices in $C_M$ except $Q$}
\end{array}
\right.\right\}.
\end{equation*}
 
\begin{lemm}
\label{invfinite2}
Let notation be as above.
We have
\[
  M_{\lambda}(b_1)=
  \begin{cases}
    \mathbb \{[e_1,e_2,e_3]\in X_*(T)_+'\mid e_1-e_3=m_1-1\} & (m_2=0) \\
    \mathbb \{[e_1,e_2,e_3]\in X_*(T)_+'\mid e_1-e_3=m_1-1, m_{\mu,\mathrm{I}}\geq -m_2\} & (m_2<0) \\
    \mathbb \{[e_1,e_2,e_3]\in X_*(T)_+'\mid e_3-e_1=m_3,\\
    \hspace{3.6cm}\max \{m_{\mu,\mathrm{I\hspace{-.1em}I}},m_{\mu,\mathrm{I\hspace{-.1em}I\hspace{-.1em}I}}\}\geq m_2\} & (m_2>0), \\
  \end{cases}
\]
\[
  M_{\lambda}(b_2)=
  \begin{cases}
    \mathbb \{[e_1,e_2,e_3]\in X_*(T)_+'\mid e_1-e_3=m_1-1\} & (m_2=1) \\
    \mathbb \{[e_1,e_2,e_3]\in X_*(T)_+'\mid e_1-e_3=m_1-1,\\
    \hspace{3.6cm} \max \{m_{\mu,\mathrm{I\hspace{-.1em}I}},m_{\mu,\mathrm{I\hspace{-.1em}I\hspace{-.1em}I}}\}\geq -m_2+1\} & (m_2<1) \\
    \mathbb \{[e_1,e_2,e_3]\in X_*(T)_+'\mid e_3-e_1=m_3, m_{\mu,\mathrm{I}}\geq m_2-1\} & (m_2>1), \\
  \end{cases}
\]
where $m_{\mu,\mathrm{I}}=\min\{e_1-e_2,e_2-e_3\}, m_{\mu,\mathrm{I\hspace{-.1em}I}}=\min\{e_1-e_2+1,e_2-e_3\}, m_{\mu,\mathrm{I\hspace{-.1em}I\hspace{-.1em}I}}=\min\{e_1-e_2,e_2-e_3+1\}$.
In particular, $M_{\lambda}(b_1)$ and $M_{\lambda}(b_2)$ are finite sets.
\end{lemm}
\begin{proof}
Set $[\mu]=[e_1,e_2,e_3]\in X_*(T)_+'$.
Let $P\in G_{[\mu]}^{C_M}(Q)$ and let $(C_0, \ldots, C_r)$ be a special gallery connecting $Q$ to $P$.
Let us first consider the case for $b_1$.
Given the relative position of $C_M$ and $C_0$, we can compute $\mathrm{inv}'(b_1 Q, P),\mathrm{inv}'(Q, b_1\sigma P)$ and $\mathrm{inv}'(P, b_1\sigma P)$ using Proposition \ref{invformula}.

If $C_M$ and $C_0$ have relative position I\hspace{-.1em}I\hspace{-.1em}I, then $\mathrm{inv}'(b_1 Q, P)=[e_1,e_2,e_3-1]$.
In this case, we can take a special gallery from $b_1 Q$ to $P$ whose first chamber is $C_M$.
To check this, note that one can take an apartment containing $C_M$ and $cl(Q,P)$ (e.g., an apartment containing $C_M$ and $P$).
Since $C_M$ and $C_0$ have relative position I\hspace{-.1em}I\hspace{-.1em}I, $C_M=b_1\sigma C_M$ and $b_1\sigma C_0$ also have relative position I\hspace{-.1em}I\hspace{-.1em}I.
So, by connecting $(C_M,\ldots, C_0,\ldots, C_r)$ and $(b_1\sigma C_0, \ldots, b_1\sigma C_r)$, we have $$\mathrm{inv}'(P,b_1\sigma P)=[e_1-e_3+1-j,j,e_3-e_1],\quad \min\{1, m_{\mu,\mathrm{I\hspace{-.1em}I\hspace{-.1em}I}}\}\le j\le m_{\mu,\mathrm{I\hspace{-.1em}I\hspace{-.1em}I}},$$
where $m_{\mu,\mathrm{I\hspace{-.1em}I\hspace{-.1em}I}}=\min\{e_1-e_2,e_2-e_3+1\}$.

If $C_M$ and $C_0$ have relative position I\hspace{-.1em}I, then $\mathrm{inv}'(Q, b_1\sigma P)=[e_1+1,e_2,e_3]$ because $C_M$ and $b_1\sigma C_0$ also have relative position I\hspace{-.1em}I.
In this case, we can take a special gallery from $Q$ to $b_1\sigma P$ whose first chamber is $C_M$.
Since $C_M$ and $C_0$ have relative position I\hspace{-.1em}I, $C_0$ and $C_M$ have relative position I\hspace{-.1em}I\hspace{-.1em}I.
So, by connecting $(C_0, \ldots, C_r)$ and $(C_M,\ldots, b_1\sigma C_0,\ldots, b_1\sigma C_r)$, we have $$\mathrm{inv}'(P,b_1\sigma P)=[e_1-e_3+1-j,j,e_3-e_1],\quad \min\{1, m_{\mu,\mathrm{I\hspace{-.1em}I}}\}\le j\le m_{\mu,\mathrm{I\hspace{-.1em}I}},$$
where $m_{\mu,\mathrm{I\hspace{-.1em}I}}=\min\{e_1-e_2+1,e_2-e_3\}$.

If $C_M$ and $C_0$ have relative position I, then $\mathrm{inv}'(Q, b_1\sigma P)=[e_1+1,e_2,e_3]$ because $C_M$ and $b_1\sigma C_0$ also have relative position I.
Let $D\neq C_M$ be the unique chamber in $cl(C_M, C_0)$ containing $Q$ and $b_2 Q$.
Then we can take a special gallery from $Q$ to $b_1\sigma P$ whose first chamber is $b_1\sigma D$.
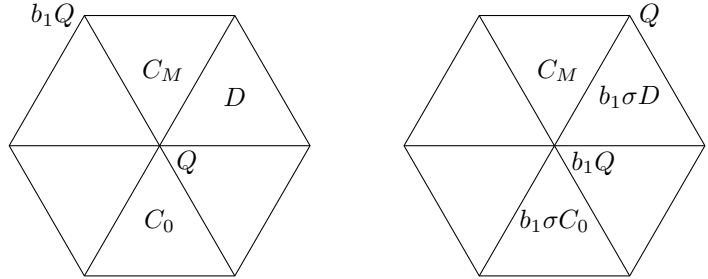
\begin{center}
\begin{tikzpicture}
 \draw (-2,0)--(-1,1.732)--(1,1.732)--(2,0)--(1,-1.732)--(-1,-1.732)--cycle;
 \draw (0,0)--(-2,0)--cycle;
 \draw (0,0)--(-1,1.732)--cycle;
 \draw (0,0)--(1,1.732)--cycle;
 \draw (0,0)--(2,0)--cycle;
 \draw (0,0)--(1,-1.732)--cycle;
 \draw (0,0)--(-1,-1.732)--cycle;
 \draw (-1,1.732)node[left]{$b_1 Q$};
 \draw (0.1,0.05)node[below right]{$Q$};
 \draw (0.05,1)node{$C_M$};
 \draw (1,0.65)node{$D$};
 \draw (0,-1)node{$C_0$};
\end{tikzpicture}
\hspace{1cm}
\begin{tikzpicture}
 \draw (-2,0)--(-1,1.732)--(1,1.732)--(2,0)--(1,-1.732)--(-1,-1.732)--cycle;
 \draw (0,0)--(-2,0)--cycle;
 \draw (0,0)--(-1,1.732)--cycle;
 \draw (0,0)--(1,1.732)--cycle;
 \draw (0,0)--(2,0)--cycle;
 \draw (0,0)--(1,-1.732)--cycle;
 \draw (0,0)--(-1,-1.732)--cycle;
 \draw (1,1.732)node[right]{$Q$};
 \draw (0.1,0.05)node[below right]{$b_1 Q$};
 \draw (0.05,1)node{$C_M$};
 \draw (1,0.65)node{$b_1\sigma D$};
 \draw (0,-1)node{$b_1\sigma C_0$};
\end{tikzpicture}
\end{center}
Using Lemma \ref{LEMMA1.1} and Lemma \ref{LEMMA1.2}, one can show that $C_0$ and $b_1\sigma D$ have relative position I or I\hspace{-.1em}I (and both of the two cases actually occur).
If $C_0$ and $b_1\sigma D$ have relative position I, then by connecting $(C_0, \ldots, C_r)$ and $(b_1\sigma D, \ldots, b_1\sigma C_0, \ldots, b_1\sigma C_r)$, we have
$$\mathrm{inv}'(P,b_1\sigma P)=[e_1-e_3+1,0,e_3-e_1].$$
If $C_0$ and $b_1\sigma D$ have relative position I\hspace{-.1em}I, then by connecting $(C_0, \ldots, C_r)$ and $(b_1\sigma D, \ldots, b_1\sigma C_0, \ldots, b_1\sigma C_r)$, we have
$$\mathrm{inv}'(P,b_1\sigma P)=[e_1-e_3+1,-j,e_3-e_1+j],\quad \min\{1,m_{\mu,\mathrm{I}}\}\le j\le m_{\mu,\mathrm{I}},$$
where $m_{\mu,\mathrm{I}}=\min\{e_1-e_2,e_2-e_3\}$.

Next, we will consider the case for $b_2$.
In the same way as the case for $b_1$, we can compute $\mathrm{inv}'(P, b_2\sigma P)$.
We will state only the results but will not give details of the proofs.
If $C_M$ and $C_0$ have relative position I\hspace{-.1em}I\hspace{-.1em}I, then by connecting $(C_0, \ldots, C_r)$ and $(C_M,\ldots, b_2\sigma C_0,\ldots, b_2\sigma C_r)$, we have
$$\mathrm{inv}'(P,b_2\sigma P)=[e_1-e_3+1,1-j,e_3-e_1+j],\quad \min\{1, m_{\mu,\mathrm{I\hspace{-.1em}I\hspace{-.1em}I}}\}\le j\le m_{\mu,\mathrm{I\hspace{-.1em}I\hspace{-.1em}I}}.$$
If $C_M$ and $C_0$ have relative position I\hspace{-.1em}I, then by connecting $(C_M,\ldots, C_0,\ldots, C_r)$ and $(b_2\sigma C_0, \ldots, b_2\sigma C_r)$, we have
$$\mathrm{inv}'(P,b_2\sigma P)=[e_1-e_3+1,1-j,e_3-e_1+j],\quad \min\{1, m_{\mu,\mathrm{I\hspace{-.1em}I}}\}\le j\le m_{\mu,\mathrm{I\hspace{-.1em}I}}.$$
If $C_M$ and $C_0$ have relative position I, then let $D'\neq C_M$ be the unique chamber in $cl(C_M, C_0)$ containing $Q$ and $b_1 Q$.
We can always take a special gallery from $Q$ to $b_2\sigma P$ whose first chamber is $b_2\sigma D'$.
Moreover, $C_0$ and $b_2\sigma D'$ have relative position I or I\hspace{-.1em}I\hspace{-.1em}I.
So if $C_0$ and $b_2\sigma D'$ have relative position I, then by connecting $(C_0, \ldots, C_r)$ and $(b_2\sigma D', \ldots, b_2\sigma C_0, \ldots, b_2\sigma C_r)$, we have
$$\mathrm{inv}'(P,b_2\sigma P)=[e_1-e_3+1,1,e_3-e_1],$$
and if $C_0$ and $b_2\sigma D'$ have relative position I\hspace{-.1em}I\hspace{-.1em}I, then by connecting $(C_0, \ldots, C_r)$ and $(b_2\sigma D', \ldots, b_2\sigma C_0, \ldots, b_2\sigma C_r)$, we have
$$\mathrm{inv}'(P,b_2\sigma P)=[e_1-e_3+1-j,1+j,e_3-e_1],\quad \min\{1,m_{\mu,\mathrm{I}}\}\le j\le m_{\mu,\mathrm{I}}.$$

Finally, the result follows from these formulas (compare the proof of Lemma \ref{invfinite}).
\end{proof}

\begin{lemm}
\label{invclosed2}
Let $Q$ be a vertex in $C_M$.
If $(X_{\lambda}(b_i)(\bar k)\cap \eta^{-1}(0))\cap G_{[\mu]}^{C_M}(Q)$ $(i=1,2)$ is non-empty, then we have 
\begin{align*}
(X_{\lambda}(b_i)(\bar k)\cap \eta^{-1}(0))\cap G_{[\mu]}^{C_M}(Q)&=(X_{\lambda}(b_i)(\bar k) \cap\eta^{-1}(0))\cap \mathcal{G}r_{[\mu]}(Q) \\
&=(X_{\lambda}(b_i)(\bar k)\cap \eta^{-1}(0))\cap \overline{\mathcal{G}r_{[\mu]}(Q)},
\end{align*}
where $\overline{\mathcal{G}r_{[\mu]}(Q)}$ denotes the closure of $\mathcal{G}r_{[\mu]}(Q)$ in $X^S$.
In particular, $(X_{\lambda}(b_i)(\bar k)\cap \eta^{-1}(0))\cap G_{[\mu]}^{C_M}(Q)$ is closed in $X_{\lambda}(b_i)(\bar k)$.
\end{lemm}
\begin{proof}
This lemma follows from the computation in Lemma \ref{invfinite2} (see the proof of Lemma \ref{invclosed}).
\end{proof}

For any $\lambda=(m_1,m_2,m_3)\in X_*(T)_+$ with $X_{\lambda}(b_i)\neq \emptyset$ $(i=1,2)$, set
\begin{equation*}
\mathcal P_{\lambda}(b_i)=
\left \{(Q,[\mu]) \left|
\begin{array}{l}
\text{$Q$ is a vertex in $C_M$ with}\\
\text{$(X_{\lambda}(b_i)(\bar k)\cap \eta^{-1}(0))\cap G_{[\mu]}^{C_M}(Q)\neq \emptyset$}\\
\end{array}
\right.\right\}.
\end{equation*}
Then for any $[\mu]=[e_1,e_2,e_3]\in M_{\lambda}(b_i)$, a tuple $(Q,[\mu])$ is contained in $\mathcal P_{\lambda}(b_i)$ if and only if $Q$ is a vertex of type $-(e_1+e_2+e_3)\in \mathbb Z/3$ in $C_M$.
Indeed, $C_M$ has the only one vertex of type $i\in \mathbb Z/3$.

\begin{prop}
\label{invdecomposition2}
For any $\lambda=(m_1,m_2,m_3)\in X_*(T)_+$ with $X_{\lambda}(b_i)\neq \emptyset$ $(i=1,2)$, we have
$$X_{\lambda}(b_i)(\bar k)\cap \eta^{-1}(0)=\bigcup_{(Q,[\mu])\in \mathcal P_{\lambda}(b_i)}((X_{\lambda}(b_i)(\bar k)\cap \eta^{-1}(0))\cap G_{[\mu]}^{C_M}(Q)).$$
\end{prop}

\begin{proof}
For any $P\in X_{\lambda}(b_i)(\bar k)\cap \eta^{-1}(0)$, we have a minimal gallery to $C_M$.
Let $Q$ be a vertex in $C_M$ such that the distance between $P$ and $Q$ is minimal.
Then $P$ is contained in $(X_{\lambda}(b_i)(\bar k)\cap \eta^{-1}(0))\cap G_{[\mu]}^{C_M}(Q)$, where $[\mu]=\mathrm{inv}'(Q,P)$.
So we obtain the decomposition in the proposition.
\end{proof}

\begin{rema}
Although the set $\mathcal P_{\lambda}(1)$ is always infinite, the sets $\mathcal P_{\lambda}(b_1)$ and $\mathcal P_{\lambda}(b_2)$ are finite.
\end{rema}

\section{Geometric Structure of the Schubert Cells}

\subsection{The Schubert Cells}
\label{theschubertcells}
The Schubert cell $Kt^{\lambda}K/K$ is locally closed in $X$ for any $\lambda \in X_*(T)_+$, so it inherits the structure of a reduced sub-ind-scheme of $X$.
We denote it by $\mathcal{G}r_{\lambda}$.

\begin{prop}
\label{Schubertcell}
Let $\lambda=(m_1,m_2,m_3),\mu=(m_1',m_2',m_3') \in X_*(T)_+$.
\begin{enumerate}[(i)]
\item The Schubert cell $\mathcal{G}r_{\lambda}$ forms an $L^{+}\mathrm{SL}_3$-orbit and is a smooth quasi-projective variety.
\item We have a canonical projection $\mathcal{G}r_{\lambda}\rightarrow \mathcal{G}r_{\mu}$ if $\lambda-\mu\in X_*(T)_+$.
\end{enumerate}
\end{prop}

\begin{proof}
Let $r=m_1+m_2+m_3$.
First note that $\mathcal{G}r_{\lambda}$ is actually a locally closed subscheme of $\mathcal{L}att^{r,(N)}$ for some sufficiently large $N$, and the left action on the projective variety $\mathcal{L}att^{r,(N)}$ of the group scheme $L^{+}\mathrm{SL}_3$ actually acts through its finite dimensional quotient $\mathrm{SL}_3(\bar k[t]/(t^{2N}))$, which is (formally) smooth.

(i) The stabilizer of $t^{\lambda}$ for the action of $L^+{\mathrm{SL}_3}$ is $L^+{\mathrm{SL}_3}\cap t^{\lambda}L^+{\mathrm{SL}_3}t^{-\lambda}$.
The induced map $$L^+{\mathrm{SL}_3}/(L^+{\mathrm{SL}_3}\cap t^{\lambda}L^+{\mathrm{SL}_3}t^{-\lambda})\rightarrow X, g\mapsto gt^{\lambda}$$ then is a locally closed embedding.
Since $\mathcal{G}r_{\lambda}(\bar k)=Kt^{\lambda}K/K=K^St^{\lambda}K/K$, the image is exactly $\mathcal{G}r_{\lambda}$.

(ii) Let us show that $L^+{\mathrm{SL}_3}\cap t^{\lambda}L^+{\mathrm{SL}_3}t^{-\lambda}$ is contained in $L^+{\mathrm{SL}_3}\cap t^{\mu}L^+{\mathrm{SL}_3}t^{-\mu}$ if and only if $\lambda-\mu\in X_*(T)_+$.
Then $\mathcal{G}r_{\lambda}\rightarrow \mathcal{G}r_{\mu}$ is the canonical quotient map
$$L^+{\mathrm{SL}_3}/(L^+{\mathrm{SL}_3}\cap t^{\lambda}L^+{\mathrm{SL}_3}t^{-\lambda})\rightarrow L^+{\mathrm{SL}_3}/(L^+{\mathrm{SL}_3}\cap t^{\mu}L^+{\mathrm{SL}_3}t^{-\mu}).$$
Let $R$ be a $k$-algebra.
Then we have $${\mathrm{SL}_3}(R[[t]])\cap t^{\lambda}\mathrm{SL}_3(R[[t]])t^{-\lambda}=\{(a_{ij})\in \mathrm{SL}_3(R[[t]])\mid \forall i<j,a_{ij}\in (t^{m_i-m_j})\}.$$
Thus ${\mathrm{SL}_3}(R[[t]])\cap t^{\lambda}\mathrm{SL}_3(R[[t]])t^{-\lambda}\subseteq {\mathrm{SL}_3}(R[[t]])\cap t^{\mu}\mathrm{SL}_3(R[[t]])t^{-\mu}$ is equivalent to saying that $m_i-m_j\geq m_i'-m_j'$ for all $i<j$, i.e., $\lambda-\mu=(m_1-m_1',m_2-m_2',m_3-m_3')\in X_*(T)_+$.
\end{proof}

Let us denote by $U$ the unipotent radical of $B$.
Let $\lambda=(m_1,m_2,m_3) \in X_*(T)_+$ and let $J^{\lambda}$ be a $k$-space defined as
$$J^{\lambda}(R)=\{(a_{ij})\in U(R[t])\mid \forall i<j, \deg a_{ij}\le m_i-m_j-1\}.$$
Then by definition, we have $J^{\lambda+M}=J^{\lambda}$, where $M=(m,m,m)\in X_*(T)_+$.
For any $\alpha\in \Phi$ and $k\in \mathbb Z$, we denote by $U_{\alpha, k}$ the image of the homomorphism $\mathbb G_a\rightarrow L\mathrm{GL}_3$ defined by $x\mapsto U_{\alpha}(t^k x)$, where $U_{\alpha}$ is the root subgroup.
Multiplication defines an isomorphism $$\prod_{\alpha\in \Phi_+,\langle \alpha, \lambda \rangle>0}\prod_{k=0}^{\langle \alpha, \lambda \rangle-1}U_{\alpha, k}\rightarrow J^{\lambda}.$$
In particular, $J^{\lambda}$ is isomorphic to the affine space of dimension $2\langle \rho,\lambda\rangle$, where $\rho$ is half the sum of the positive roots.
We will often write $U_{ij,k}$ instead of $U_{\chi_{ij},k}$.

\begin{lemm}
The morphism $J^{\lambda}\rightarrow \mathcal{G}r_{\lambda}$ defined by $g\mapsto gt^{\lambda}$ is an open immersion.
Moreover, $\mathcal{G}r_{\lambda}$ is irreducible and of dimension $2\langle \rho,\lambda\rangle$.
\end{lemm}

\begin{proof}
See \cite[Lemme 2.2]{NP}.
\end{proof}
From now on, we see $J^{\lambda}$ as an open subscheme of $\mathcal{G}r_{\lambda}$ by this open immersion.

By Proposition \ref{Schubertcell}, there exists a canonical projection $\mathcal{G}r_{\lambda}\rightarrow \mathcal{G}r_{(1,0,0)}$ (resp.\ $\mathcal{G}r_{\lambda}\rightarrow \mathcal{G}r_{(0,0,-1)}$) if $m_1>m_2$ (resp.\ $m_2>m_3$).
To shorten notation we set $\mathcal{G}r_1=\mathcal{G}r_{(1,0,0)}, \mathcal{G}r_{-1}=\mathcal{G}r_{(0,0,-1)}, J^{1}=J^{(1,0,0)}, J^{-1}=J^{(0,0,-1)}$.
Let $\mathop{\mathrm{Flag}}$ be the reduced closed subscheme of $\mathcal{G}r_1\times \mathcal{G}r_{-1}$ defined as $\mathop{\mathrm{Flag}}(\bar k)=\{(\mathscr L, \mathscr L')\in\mathcal{G}r_1(\bar k)\times \mathcal{G}r_{-1}(\bar k)\mid \mathscr L\supset t\mathscr L'\}$.
Then $\mathop{\mathrm{Flag}}$ can be covered by open subsets isomorphic to the $3$-dimensional affine space.
In particular, we have an isomorphism $$U_{12,0}\times U_{23,0}\times U_{13,0}\cong (J^1\times J^{-1})\cap \mathop{\mathrm{Flag}}.$$

Let $E,S$ be $\bar k$-schemes.
Then a morphism $p\colon E\rightarrow S$ is an {\it affine bundle of rank $n$} over $S$ if $S$ has an open covering by $U_i$, and there are isomorphisms
$$p^{-1}(U_i)\cong U_i\times \mathbb A^n$$
such that $p$ restricted to $p^{-1}(U_i)$ corresponds to the projection from $U_i\times \mathbb A^n$ to $U_i$.

\begin{lemm}
\label{affinebundle}
Let $\lambda=(m_1,m_2,m_3)\in X_*(T)_+$.
\begin{enumerate}[(i)]
\item If $m_1>m_2>m_3$, then the canonical projection $\mathcal{G}r_{\lambda}\rightarrow \mathcal{G}r_1\times \mathcal{G}r_{-1}$ factors through $\mathop{\mathrm{Flag}}$.
Then $\varphi_{\lambda}\colon \mathcal{G}r_{\lambda}\rightarrow \mathop{\mathrm{Flag}}$ is an affine bundle of rank $2(m_1-m_3-1)-1$.
In particular, we have an isomorphism
$$\varphi_{\lambda}^{-1}((J^1\times J^{-1})\cap \mathop{\mathrm{Flag}})=J^{\lambda}\cong ((J^1\times J^{-1})\cap \mathop{\mathrm{Flag}})\times \mathbb A^{2(m_1-m_3-1)-1}$$
such that $\varphi_{\lambda}$ restricted to $\varphi_{\lambda}^{-1}((J^1\times J^{-1})\cap \mathop{\mathrm{Flag}})$ corresponds to the projection from $((J^1\times J^{-1})\cap \mathop{\mathrm{Flag}})\times \mathbb A^{2(m_1-m_3-1)-1}$ to $(J^1\times J^{-1})\cap \mathop{\mathrm{Flag}}$.

\item If $m_1>m_2=m_3 $, then the canonical projection $\varphi_{\lambda}\colon \mathcal{G}r_{\lambda}\rightarrow \mathcal{G}r_1\cong \mathbb P^2$ is an affine bundle of rank $2(m_1-m_3-1)$.
In particular, we have an isomorphism
$$\varphi_{\lambda}^{-1}(J^1)=J^{\lambda}\cong J^1\times \mathbb A^{2(m_1-m_3-1)}$$
such that $\varphi_{\lambda}$ restricted to $\varphi_{\lambda}^{-1}(J^1)$ corresponds to the projection from $J^1\times \mathbb A^{2(m_1-m_3-1)}$ to $J^1$.

\item If $m_1=m_2>m_3$, then the canonical projection $\varphi_{\lambda}\colon \mathcal{G}r_{\lambda}\rightarrow \mathcal{G}r_{-1}\cong \mathbb P^2$ is an affine bundle of rank $2(m_1-m_3-1)$.
In particular, we have an isomorphism
$$\varphi_{\lambda}^{-1}(J^{-1})=J^{\lambda}\cong J^{-1}\times \mathbb A^{2(m_1-m_3-1)}$$
such that $\varphi_{\lambda}$ restricted to $\varphi_{\lambda}^{-1}(J^{-1})$ corresponds to the projection from $J^{-1}\times \mathbb A^{2(m_1-m_3-1)}$ to $J^{-1}$.
\end{enumerate}
\end{lemm}

\begin{proof}
Since $\mathcal{G}r_{\lambda+M}\cong \mathcal{G}r_{\lambda}$, where $M=(m, m, m)$, we may assume $m_2=0$.
Then let us write 
\begin{gather*}
\phi_{m_1,m_3}\colon \mathcal{G}r_{(m_1+1,0,m_3)}\rightarrow \mathcal{G}r_{(m_1,0,m_3)}\quad (m_1>0), \\
\psi_{m_1,m_3}\colon \mathcal{G}r_{(m_1,0,m_3-1)}\rightarrow \mathcal{G}r_{(m_1,0,m_3)}\quad (m_3<0)
\end{gather*}
for the canonical projections.
First, let us show that (1) $J^{\lambda}=\varphi_{\lambda}^{-1}((J^1\times J^{-1})\cap \mathop{\mathrm{Flag}})$ (2) $J^{\lambda}=\varphi_{\lambda}^{-1}(J^1)$ (3) $J^{\lambda}=\varphi_{\lambda}^{-1}(J^{-1})$ corresponding to the equation in (i), (ii) and (iii) respectively.
For this, it suffices to show 
\begin{gather*}
J^{(1,0,-1)}=\varphi_{(1,0,-1)}^{-1}((J^1\times J^{-1})\cap \mathop{\mathrm{Flag}}), \\
J^{(m_1+1,0,m_3)}=\phi_{m_1,m_3}^{-1}(J^{\lambda}),\quad J^{(m_1,0,m_3-1)}=\psi_{m_1,m_3}^{-1}(J^{\lambda})
\end{gather*}
where $\lambda=(m_1,0,m_3)$.
Indeed, any $\varphi_{\lambda}$ can be obtained as the composite of some $\phi_{m_1',m_3'},\psi_{m_1',m_3'},$ and $\varphi_{(1,0,-1)}$.

Obviously, we have $\phi_{m_1,m_3}(J^{(m_1+1,0,m_3)})\subseteq J^{\lambda}$.
To see $J^{(m_1+1,0,m_3)}=\phi_{m_1,m_3}^{-1}(J^{\lambda})$, assume that $\phi_{m_1,m_3}(\bar g)\in J^{\lambda}(\bar k)$ for a matrix $g\in \mathrm{SL}_3(\mathcal O)$.
This is equivalent to saying that there exist matrices $v\in J^{\lambda}(\bar k)$ and $t^{\lambda}at^{-\lambda}\in \mathrm{SL}_3(\mathcal O)\cap t^{\lambda}\mathrm{SL}_3(\mathcal O)t^{-\lambda}$ such that $$g=vt^{\lambda}at^{-\lambda}.$$
Then it suffices to show that for any such $g$, there exists a matrix
$$u=
\begin{pmatrix}
1 & u_{12}t^{m_1} & u_{13}t^{m_1-m_3} \\
0 & 1 & 0 \\
0 & 0 & 1 \\
\end{pmatrix}
\in U_{12,m_1}\times U_{13,m_1-m_3}, u_{12},u_{13}\in \bar k$$
satisfying $$u^{-1}v^{-1}g=u^{-1}t^{\lambda}at^{-\lambda}\in \mathrm{SL}_3(\mathcal O)\cap t^{(m_1+1,0,m_3)}\mathrm{SL}_3(\mathcal O)t^{(-m_1-1,0,-m_3)}.$$
Let $a=(a_{ij})=(a_{ij}(t))$.
Then this condition holds if and only if 
\begin{gather*}
(u^{-1}t^{\lambda}at^{-\lambda})_{12}=t^{m_1}(a_{12}-u_{12}a_{22}-u_{13}a_{32})\in (t^{m_1+1})\ \text{and}\\
(u^{-1}t^{\lambda}at^{-\lambda})_{13}=t^{m_1-m_3}(a_{13}-u_{12}a_{23}-u_{13}a_{33})\in (t^{m_1-m_3+1}).
\end{gather*}
Since $t^{\lambda}at^{-\lambda}\in \mathrm{SL}_3(\mathcal O)\cap t^{\lambda}\mathrm{SL}_3(\mathcal O)t^{-\lambda}$ and $m_1>0$, $a_{21}$ and $a_{31}$ are not units.
So the cofactor expansion of $a$ along the first column implies that the determinant of the matrix
$$A(t)=
\begin{pmatrix}
a_{22}(t) & a_{32}(t) \\
a_{23}(t) & a_{33}(t) \\
\end{pmatrix}
$$
is a unit.
Equivalently we have $\det A(0)\neq 0$, and then we can find a solution $(u_{12}, u_{13})$ of the conditions $a_{12}-u_{12}a_{22}-u_{13}a_{32}, a_{13}-u_{12}a_{23}-u_{13}a_{33}\in (t)$, i.e., 
$$A(0)
\begin{pmatrix}
u_{12}  \\
u_{13}  \\
\end{pmatrix}
=
\begin{pmatrix}
a_{12}(0)  \\
a_{13}(0)  \\
\end{pmatrix}.
$$
Thus we get $J^{(m_1+1,0,m_3)}=\phi_{m_1,m_3}^{-1}(J^{\lambda})$ by what we have just proven.
The proof for $$J^{(m_1,0,m_3-1)}=\psi_{m_1,m_3}^{-1}(J^{\lambda})$$ is similar ($m_3<0$ yields $a_{33}(0)\neq 0$, and then we may find the solutions $u_{13},u_{23}$).

Next we must show $J^{(1,0,-1)}=\varphi_{(1,0,-1)}^{-1}((J^1\times J^{-1})\cap \mathop{\mathrm{Flag}})$.
Assume that a matrix $g\in \mathrm{SL}_3(\mathcal O)$ satisfies $\varphi_{(1,0,-1)}(\bar g)\in J^1(\bar k)\times J^{-1}(\bar k)$.
Let $\lambda_1=(1,0,0), \lambda_{-1}=(0,0,-1)$.
Then $\varphi_{(1,0,-1)}(g)\in J^1(\bar k)\times J^{-1}(\bar k)$ is equivalent to saying that there exist matrices $v_i\in J^{i}(\bar k),t^{\lambda_i}a_{\lambda_i}t^{\lambda_i}\in \mathrm{SL}_3(\mathcal O)\cap t^{\lambda_i}\mathrm{SL}_3(\mathcal O)t^{-\lambda_i}$ $(i=\pm 1)$ such that $$g=v_1t^{\lambda_1}a_{\lambda_1}t^{-\lambda_1}=v_{-1}t^{\lambda_{-1}}a_{\lambda_{-1}}t^{-\lambda_{-1}}.$$
Let us write 
$$v_1=
\begin{pmatrix}
1 & z_1 & z_4 \\
0 & 1 & 0 \\
0 & 0 & 1 \\
\end{pmatrix},\quad
v_{-1}=
\begin{pmatrix}
1 & 0 & z_3 \\
0 & 1 & z_2 \\
0 & 0 & 1 \\
\end{pmatrix},\quad
z_1, z_2, z_3, z_4\in \bar k.
$$
Further, set $a_{\lambda_{-1}}=(a_{ij})=(a_{ij}(t))$ and
$$u=
\begin{pmatrix}
1 & z_1 & a_{33}(0)^{-1}(a_{13}(0)-z_1a_{23}(0))t \\
0 & 1 & 0 \\
0 & 0 & 1 \\
\end{pmatrix}.
$$
Note that we have $a_{32}(0)=0$ and $a_{33}(0)\neq 0$ because $t^{\lambda_{-1}}a_{\lambda_{-1}}t^{\lambda_{-1}}\in \mathrm{SL}_3(\mathcal O)\cap t^{\lambda_{-1}}\mathrm{SL}_3(\mathcal O)t^{-\lambda_{-1}}$.
Then $(v_1)^{-1}v_{-1}=t^{\lambda_1}a_{\lambda_1}t^{-\lambda_1}(t^{\lambda_{-1}}a_{\lambda_{-1}}t^{-\lambda_{-1}})^{-1}$ yields $z_4=z_3-z_1z_2$ (resp.\ $a_{12}(0)-z_1a_{22}(0)=0$) by comparing the (1,3) (resp.\ (1,2)) entry of the matrices.
In particular, we have $\varphi_{(1,0,-1)}(J^{(1,0,-1)})\subseteq (J^1\times J^{-1})\cap \mathop{\mathrm{Flag}}$.
Moreover, using these equations, one can check that $$u^{-1}(v_{-1})^{-1}g\in \mathrm{SL}_3(\mathcal O)\cap t^{(1,0,-1)}\mathrm{SL}_3(\mathcal O)t^{(-1,0,1)}.$$
Thus we get $J^{(1,0,-1)}=\varphi_{(1,0,-1)}^{-1}((J^1\times J^{-1})\cap \mathop{\mathrm{Flag}})$.

We can cover $\mathop{\mathrm{Flag}}$ (resp.\ $\mathcal{G}r_1$, resp.\ $\mathcal{G}r_{-1}$) by suitable open subvarieties isomorphic to $U=(J^1\times J^{-1})\cap \mathop{\mathrm{Flag}}$ (resp.\ $J^1$, resp.\ $J^{-1}$).
Indeed, there exists a finite set $\{g_i\}_i$ with $g_i\in \mathrm{SL}_3(\bar k)$ such that $\{g_iU\}_i$ is an open covering (for example, we can take the set of standard representatives of the finite Weyl group of $\mathrm{SL}_3$).
Since each $J^{\lambda}$ is a product of root subgroups,
one easily verifies that such an open covering defines the structure of an affine bundle, and its relative dimension is $2(m_1-m_3-1)-1$ (resp.\ $2(m_1-m_3-1)$, resp.\ $2(m_1-m_3-1)$).
\end{proof}

\begin{rema}
For $\lambda\in X_*(T)_+$, there is a natural projection $$\mathcal{G}r_{\lambda}\rightarrow \mathrm{SL}_n/P_{\lambda}$$ induced by $L^+\mathrm{SL}_n\rightarrow \mathrm{SL}_n, t\mapsto 0$, where $P_{\lambda}$ is the stabilizer of $t^{\lambda}$ in $\mathrm{SL}_n$.
This projection is actually an affine bundle, and the lemma is the special case of this fact (see for example \cite[Section 2]{MV} that treats the affine Grassmannian over $\mathbb C$).
\end{rema}

Note that the set $\mathcal{G}r_{[\lambda]}(Q)$ (Section \ref{b=1}) can be identified with the Schubert cell $\mathcal{G}r_{\lambda}(\bar k)$ for any vertex $Q$ in $\mathcal B_1$ and $\lambda \in X_*(T)_+$. 
So we can see this set as a variety, and denote it also by $\mathcal{G}r_{[\lambda]}(Q)$ (of course, this definition is independent of the choice of $\lambda$).
Let $\mu\in X_*(T)_+$ with $\lambda-\mu\in X_*(T)_+$.
Via this identification, the projection (Proposition \ref{Schubertcell}) $$\varphi_{\lambda, \mu}\colon \mathcal{G}r_{\lambda}(\bar k)\rightarrow \mathcal{G}r_{\mu}(\bar k)$$ can also be described as follows.
Let $P\in \mathcal{G}r_{[\lambda]}(Q)$.
Then there exists a unique vertex in $\mathcal{G}r_{[\mu]}(Q)\cap cl(Q,P)$.
So we have a map $$\mathcal{G}r_{[\lambda]}(Q)\rightarrow \mathcal{G}r_{[\mu]}(Q),$$
which sends $P$ to the unique vertex in $\mathcal{G}r_{[\mu]}(Q)\cap cl(Q,P)$.
This map corresponds to $\varphi_{\lambda, \mu}$ through the identification.
For example, if $Q=[\Lambda_{\bar k}]$ and $P=[gt^{\lambda}\Lambda_{\bar k}]$ with $g\in \mathrm{SL}_3(\mathcal O)$, then the unique vertex in $\mathcal{G}r_{[\mu]}([\Lambda_{\bar k}])\cap cl([\Lambda_{\bar k}],P)$ is $[gt^{\mu}\Lambda_{\bar k}]$.
Indeed, $cl([\Lambda_{\bar k}], [t^{\lambda}\Lambda_{\bar k}])$ is isomorphic to $cl([\Lambda_{\bar k}], [gt^{\lambda}\Lambda_{\bar k}])$ by multiplication with $g$.
In particular, $\varphi_{\lambda}$ (Lemma \ref{affinebundle}) can be seen as a morphism mapping $P\in \mathcal{G}r_{[\lambda]}(Q)$ to the first alcove or vertex of $cl(Q, P)$.

\subsection{Subvarieties of the Schubert Cells}
\label{subvarieties}
We can identify $\mathop{\mathrm{Flag}}(\bar k)$ with the set of chambers containing $[\Lambda_{\bar k}]$.
Then we define a locally closed subvariety $X_u$ of $\mathop{\mathrm{Flag}}$ by 
$$X_{u}(\bar k)=\{C\in \mathop{\mathrm{Flag}}(\bar k)\mid \text{$C$ and $\sigma C$ have relative position $u$}\},$$
where $u=$ I, I\hspace{-.1em}I or I\hspace{-.1em}I\hspace{-.1em}I.
Let $W$ be the finite Weyl group of $\mathrm{GL}_3$.
Obviously, if $u=$ I (resp.\ I\hspace{-.1em}I, resp.\ I\hspace{-.1em}I\hspace{-.1em}I), then $X_u$ is the classical Deligne-Lusztig variety associated with the maximal length (resp.\ a Coxeter, resp.\  a Coxeter) element in $W$.
In particular, if $u=$ I\hspace{-.1em}I or I\hspace{-.1em}I\hspace{-.1em}I, then $X_u$ can be identified with the Drinfeld upper half space (of dimension $2$)
$$\Omega=\mathbb P^2 \setminus \bigcup_{H\in \mathcal H} H,$$
where $\mathcal H$ is the set of $k$-rational hyperplanes in $\mathbb P^2$.

Next we introduce subsets of $\mathcal{G}r_{[\mu]}(Q)$ for each vertex $Q$ in $\mathcal B_1$ (Section \ref{ADLVset}).
For any $[\mu]=[e_1,e_2,e_3]\in X_*(T)'_+$, set
\begin{equation*}
\mathcal{G}r_{[\mu]}^{\mathrm{I}}(Q)=
\left \{P\in \mathcal{G}r_{[\mu]}(Q) \left|
\begin{array}{l}
\text{there exists a first chamber $C$ of $cl(Q,P)$}\\
\text{such that $C$ and $\sigma C$ have relative position I}
\end{array}
\right.\right\}.
\end{equation*}
In case $e_1=e_2$ or $e_2=e_3$, then $P\in \mathcal{G}r_{[\mu]}(Q)$ belongs to $\mathcal{G}r_{[\mu]}^{\mathrm{I}}(Q)$ if and only if the first ``vertex" of $cl(Q,P)$ is not contained in $\mathcal B_1$.
Indeed, let $P_0$ be the first vertex of $cl(Q, P)$ which is not contained in $\mathcal B_1$.
Then chambers $C=\{Q, P_0, P_1\}$ and $\sigma C$ have relative position I\hspace{-.1em}I or I\hspace{-.1em}I\hspace{-.1em}I if and only if $P_1$ or $\sigma P_1\in cl(P_0, \sigma P_0)$.
So if we take $P_1$ such that $P_1,\sigma P_1\notin cl(P_0, \sigma P_0)$, then $C$ and $\sigma C$ have relative position I.
Further, for any $[\mu]=[e_1,e_2,e_3] \in X_*(T)'_+$ with $e_1>e_2>e_3$, set
\begin{equation*}
\mathcal{G}r_{[\mu]}^{u}(Q)=
\left \{P\in \mathcal{G}r_{[\mu]}(Q) \left|
\begin{array}{l}
\text{$C$ and $\sigma C$ have relative position $u$, where}\\
\text{$C$ is the unique first chamber of $cl(Q,P)$}\\
\end{array}
\right.\right\},
\end{equation*}
where $u=$ I\hspace{-.1em}I or I\hspace{-.1em}I\hspace{-.1em}I.

For any $P\in \mathcal{G}r_{[\mu]}^{u}(Q)$, let $(C_0,\ldots, C_r)$ be the special gallery connecting $Q$ to $P$, and let $1\le j<m_{\mu}, m_{\mu}=\min\{e_1-e_2, e_2-e_3\}$.
Let $a=2(e_1-e_2-1)$.
If $u=$ I\hspace{-.1em}I, then we define the sets
\begin{equation*}
\mathcal{G}r_{[\mu]}^{\mathrm{I\hspace{-.1em}I},j}(Q)=
\left \{P\in \mathcal{G}r_{[\mu]}^{\mathrm{I\hspace{-.1em}I}}(Q) \left|
\begin{array}{l}
\text{$\sigma C_{a+2i}\subset cl(C_{2i}, \sigma C_{a+2i-1})$ for $1\le i<j$, and}\\
\text{$\sigma C_{a+2j}\nsubseteq cl(C_{2j}, \sigma C_{a+2j-1})$}\\
\end{array}
\right.\right\}
\end{equation*}
and
$$
\mathcal{G}r_{[\mu]}^{\mathrm{I\hspace{-.1em}I},m_{\mu}}(Q)=
\{P\in \mathcal{G}r_{[\mu]}^{\mathrm{I\hspace{-.1em}I}}(Q)\mid 
\text{$\sigma C_{a+2i}\subset cl(C_{2i}, \sigma C_{a+2i-1})$ for $1\le i< m_{\mu}$}\}.
$$
If $u=$ I\hspace{-.1em}I\hspace{-.1em}I, then we define the sets
\begin{equation*}
\mathcal{G}r_{[\mu]}^{\mathrm{I\hspace{-.1em}I\hspace{-.1em}I},j}(Q)=
\left \{P\in \mathcal{G}r_{[\mu]}^{\mathrm{I\hspace{-.1em}I\hspace{-.1em}I}}(Q) \left|
\begin{array}{l}
\text{$C_{a+2i}\subset cl(C_{a+2i-1}, \sigma C_{2i})$ for $1\le i<j$, and}\\
\text{$C_{a+2j}\nsubseteq cl(C_{a+2j-1}, \sigma C_{2j})$}\\
\end{array}
\right.\right\}
\end{equation*}
and
$$
\mathcal{G}r_{[\mu]}^{\mathrm{I\hspace{-.1em}I\hspace{-.1em}I},m_{\mu}}(Q)=
\{P\in \mathcal{G}r_{[\mu]}^{\mathrm{I\hspace{-.1em}I\hspace{-.1em}I}}(Q)\mid 
\text{$C_{a+2i}\subset cl(C_{a+2i-1}, \sigma C_{2i})$ for $1\le i< m_{\mu}$}\}.
$$

Recall that $\mathcal{G}r_{[\mu]}(Q)$ can be identified with the Schubert cell $\mathcal{G}r_{\mu}$ for any vertex $Q$ in $\mathcal B_1$ and $\mu \in X_*(T)_+$.
Moreover, the subsets of $\mathcal{G}r_{[\mu]}(Q)$ defined above can be seen as locally closed reduced subvarieties of the variety $\mathcal{G}r_{[\mu]}(Q)$, and are denoted by the same symbols (see the next proposition).

\begin{prop}
\label{geometricstructure}
Let notation be as above.
\begin{enumerate}[(i)]
\item Let $[\mu]=[e_1,e_2,e_3]\in X_*(T)'_+$.
If $e_1=e_2>e_3$ or $e_1>e_2=e_3$ (resp.\  $e_1>e_2>e_3$),
then the $\bar k$-variety $\mathcal{G}r_{[\mu]}^{\mathrm{I}}(Q)$ is an affine bundle of rank $2(e_1-e_3)-2$ (resp.\ $2(e_1-e_3)-3$) over $\mathbb P^2\setminus \mathbb P^2(k)$ (resp.\ $X_{\mathrm{I}}$).
\item Let $[\mu]=[e_1,e_2,e_3] \in X_*(T)'_+$ with $e_1>e_2>e_3$.
Then the $\bar k$-variety $\mathcal{G}r_{[\mu]}^{\mathrm{I\hspace{-.1em}I}}(Q)$ (resp.\ $\mathcal{G}r_{[\mu]}^{\mathrm{I\hspace{-.1em}I},j}(Q)$, resp.\ $\mathcal{G}r_{[\mu]}^{\mathrm{I\hspace{-.1em}I},m_{\mu}}(Q)$) is contained in $J^{\mu}$ via the identification above, and isomorphic to $\Omega\times \mathbb A^{2(e_1-e_3)-3}$ (resp.\ $\Omega\times \mathbb G_m\times \mathbb A^{2(e_1-e_3)-j-3}$, resp.\ $\Omega\times \mathbb A^{2(e_1-e_3)-m_{\mu}-2}$).
\item Let $[\mu]=[e_1,e_2,e_3] \in X_*(T)'_+$ with $e_1>e_2>e_3$.
Then the $\bar k$-variety $\mathcal{G}r_{[\mu]}^{\mathrm{I\hspace{-.1em}I\hspace{-.1em}I}}(Q)$ (resp.\ $\mathcal{G}r_{[\mu]}^{\mathrm{I\hspace{-.1em}I\hspace{-.1em}I},j}(Q)$, resp.\ $\mathcal{G}r_{[\mu]}^{\mathrm{I\hspace{-.1em}I\hspace{-.1em}I},m_{\mu}}(Q)$) is contained in $J^{\mu}$ via the identification above, and isomorphic to $\Omega\times \mathbb A^{2(e_1-e_3)-3}$ (resp.\ $\Omega\times \mathbb G_m\times \mathbb A^{2(e_1-e_3)-j-3}$, resp.\ $\Omega\times \mathbb A^{2(e_1-e_3)-m_{\mu}-2}$).
\end{enumerate}
\end{prop}

\begin{proof}
It suffices to prove the case for $Q=[\Lambda_{\bar k}]$.
We omit $[\Lambda_{\bar k}]$ from the notation (for instance, $\mathcal{G}r_{[\mu]}^{\mathrm{I}}=\mathcal{G}r_{[\mu]}^{\mathrm{I}}([\Lambda_{\bar k}])$).
Further, we set $\mu=(e_1, e_2, e_3)\in X_*(T)_+$.
As explained in the last part of Section \ref{theschubertcells}, for any $[\mu]=[e_1,e_2,e_3]\in X_*(T)'_+$ with $e_1>e_2>e_3$ (resp.\ $e_1=e_2>e_3$ or $e_1>e_2=e_3$), we have the affine bundle
$$\text{$\varphi_{\mu}\colon \mathcal{G}r_{[\mu]}\rightarrow \mathop{\mathrm{Flag}}$ (resp.\ $\varphi_{\mu}\colon \mathcal{G}r_{[\mu]}\rightarrow \mathbb P^2$)}.$$
So (i) follows immediately from the facts explained before the proposition.
Let us prove (ii).
The proof for (iii) is similar.

First note that $X_{\mathrm{I\hspace{-.1em}I}}\subset \mathop{\mathrm{Flag}}$ is actually contained in $(J^1\times J^{-1})\cap \mathop{\mathrm{Flag}}$.
Indeed, for any $C\in X_{\mathrm{I\hspace{-.1em}I}}(\bar k)$ and the vertex $P_i$ $(i=1,2)$ of type $i$ in $C$, the only vertex in $\mathcal B_1$ contained in $cl(P_i, \sigma P_i)$ is $[\Lambda_{\bar k}]$.
In particular, both of the sets $\{[\Lambda_{\bar k}], P_1, [\mathcal O\oplus t\mathcal O\oplus t\mathcal O]\}$ and $\{[\Lambda_{\bar k}], P_2, [\mathcal O\oplus \mathcal O\oplus t\mathcal O]\}$ are not chambers, so $P_1\in J^1(\bar k)$ and $P_2\in J^{-1}(\bar k)$.

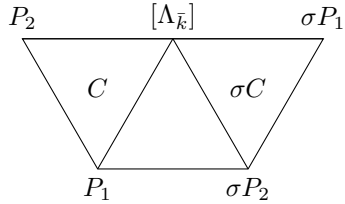
\begin{center}
\begin{tikzpicture}
 \draw (-2,0)--(2,0)--(1,-1.732)--(-1,-1.732)--cycle;
 \draw (0,0)--(-2,0)--cycle;
 \draw (0,0)--(2,0)--cycle;
 \draw (0,0)--(1,-1.732)--cycle;
 \draw (0,0)--(-1,-1.732)--cycle;
 \draw (0,0)node[above]{$[\Lambda_{\bar k}]$};
 \draw (-1,-1.732)node[below]{$P_1$};
 \draw (2,0)node[above]{$\sigma P_1$};
 \draw (1,-1.732)node[below]{$\sigma P_2$};
 \draw (-2,0)node[above]{$P_2$};
 \draw (-1,-0.65)node{$C$};
 \draw (1,-0.65)node{$\sigma C$};
\end{tikzpicture}
\end{center}
Then, by Lemma \ref{affinebundle}, we have $\mathcal{G}r_{[\mu]}^{\mathrm{I\hspace{-.1em}I}}\subset J^{\mu}$ and
 $$\mathcal{G}r_{[\mu]}^{\mathrm{I\hspace{-.1em}I}}=\varphi_{\mu}^{-1}(X_{\mathrm{I\hspace{-.1em}I}})\cong X_{\mathrm{I\hspace{-.1em}I}}\times \mathbb A^{2(e_1-e_3)-3}\cong \Omega \times \mathbb A^{2(e_1-e_3)-3}$$
as a locally closed subvariety of $\mathcal{G}r_{[\mu]}$.

Next, we show the statement for $\mathcal{G}r_{[\mu]}^{\mathrm{I\hspace{-.1em}I},m_{\mu}}$ by induction on $m_{\mu}=\min\{e_1-e_2, e_2-e_3\}$.
If $m_{\mu}=1$, then $\mathcal{G}r_{[\mu]}^{\mathrm{I\hspace{-.1em}I},m_{\mu}}=\mathcal{G}r_{[\mu]}^{\mathrm{I\hspace{-.1em}I}}$ and the statement follows from the isomorphism above.
Let us suppose that $m_{\mu}\geq 2$ and that the result holds for every $\mathcal{G}r_{[\nu]}^{\mathrm{I\hspace{-.1em}I},m_{\nu}}$ with $m_{\nu}<m_{\mu}$.
If $e_2-e_3=m_{\mu}$, we have a canonical projection 
$\mathcal{G}r_{[\mu]}^{\mathrm{I\hspace{-.1em}I}}\rightarrow \mathcal{G}r_{[\mu']}^{\mathrm{I\hspace{-.1em}I}}$, where $[\mu']=[2e_2-e_3,e_2,e_3]$.
Then the inverse image of $\mathcal{G}r_{[\mu']}^{\mathrm{I\hspace{-.1em}I},m_{\mu'}}$ by this projection is $\mathcal{G}r_{[\mu]}^{\mathrm{I\hspace{-.1em}I},m_{\mu}}$.
To check this, let $P\in \mathcal{G}r_{[\mu]}^{\mathrm{I\hspace{-.1em}I}}$ and let $P_0$ be the unique vertex in $\mathcal{G}r_{[\mu']}^{\mathrm{I\hspace{-.1em}I}}$ which is also contained in $cl([\Lambda_{\bar k}], P)$.
Let $(C_0, \ldots, C_r)$ (resp.\ $(D_0, \ldots, D_s)$) be the special gallery connecting $[\Lambda_{\bar k}]$ to $P_0$ (resp.\ $P$), and let $a_1=2(e_2-e_3-1)$ (resp.\ $a_2=2(e_1-e_2-1)$).
Fix an apartment containing $cl([\Lambda_{\bar k}], P)$, and let $\rho_1$ (resp.\ $\rho_2$) be the retraction of $\mathcal B_{\infty}$ onto this apartment with center $C_r$ (resp.\ $D_s$).
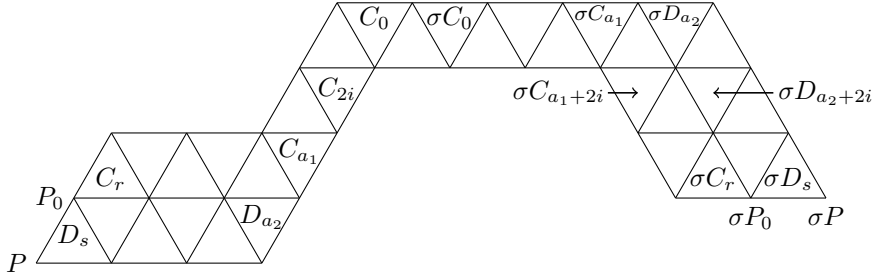
\begin{center}
\begin{tikzpicture}
 \draw (3,0)--(4, 0)--(5,0)--cycle;
 \draw (-4.5,0.866)--(-3.5,0.866)--(-2.5,0.866)--(-1.5,0.866)--cycle;
 \draw (2.5,0.866)--(3.5,0.866)--(4.5,0.866)--cycle;
 \draw (3,0)--(2.5,0.866)--cycle;
 \draw (3,0)--(3.5,0.866)--cycle;
 \draw (4,0)--(3.5,0.866)--cycle;
 \draw (4,0)--(4.5,0.866)--cycle;
 \draw (5,0)--(4.5,0.866)--cycle;
 \draw (-2,1.732)--(-1,1.732)--(0,1.732)--(1,1.732)--(2,1.732)--(3,1.732)--(4,1.732)--cycle;
 \draw (-5.5,-0.866)--(-4.5,-0.866)--(-3.5,-0.866)--(-2.5,-0.866)--cycle;
 \draw (-2,0)--(-2.5,-0.866)--cycle;
 \draw (-3,0)--(-2.5,-0.866)--cycle;
 \draw (-3,0)--(-3.5,-0.866)--cycle;
 \draw (-4,0)--(-3.5,-0.866)--cycle;
 \draw (-4,0)--(-4.5,-0.866)--cycle;
 \draw (-5,0)--(-4.5,-0.866)--cycle;
 \draw (-5,0)--(-5.5,-0.866)--cycle;
 \draw (-5,0)--(-4.5,0.866)--cycle;
 \draw (-5,0)--(-4,0)--(-3,0)--(-2,0)--cycle;
 \draw (-4,0)--(-3.5,0.866)--cycle;
 \draw (-3,0)--(-3.5,0.866)--cycle;
 \draw (-3,0)--(-2.5,0.866)--cycle;
 \draw (-2,0)--(-2.5,0.866)--cycle;
 \draw (-2,0)--(-1.5,0.866)--cycle;
 \draw (-4,0)--(-4.5,0.866)--cycle;
 \draw (-2,1.732)--(-2.5,0.866)--cycle;
 \draw (-2,1.732)--(-1.5,0.866)--cycle;
 \draw (-1,1.732)--(-1.5,0.866)--cycle;
 \draw (2,1.732)--(2.5,0.866)--cycle;
 \draw (3,1.732)--(2.5,0.866)--cycle;
 \draw (3,1.732)--(3.5,0.866)--cycle;
 \draw (4,1.732)--(3.5,0.866)--cycle;
 \draw (4,1.732)--(4.5,0.866)--cycle;
 \draw (-1.5,2.598)--(-0.5,2.598)--(1.5,2.598)--(2.5,2.598)--(3.5,2.598)--cycle;
 \draw (-2,1.732)--(-1.5,2.598)--cycle;
 \draw (-1,1.732)--(-1.5,2.598)--cycle;
 \draw (-1,1.732)--(-0.5,2.598)--cycle;
 \draw (0,1.732)--(-0.5,2.598)--cycle;
 \draw (0,1.732)--(0.5,2.598)--cycle;
 \draw (1,1.732)--(0.5,2.598)--cycle;
 \draw (1,1.732)--(1.5,2.598)--cycle;
 \draw (2,1.732)--(1.5,2.598)--cycle;
 \draw (2,1.732)--(2.5,2.598)--cycle;
 \draw (3,1.732)--(2.5,2.598)--cycle;
 \draw (3,1.732)--(3.5,2.598)--cycle;
 \draw (4,1.732)--(3.5,2.598)--cycle;
 \draw (-1,2.4)node{$C_0$};
 \draw (-1.5,1.45)node{$C_{2i}$};
 \draw (-2,0.62)node{$C_{a_1}$};
 \draw (-2.5,-0.25)node{$D_{a_2}$};
 \draw (-4.51,0.25)node{$C_r$};
 \draw (-5,-0.5)node{$D_s$};
 \draw (0,2.4)node{$\sigma C_0$};
 \draw (2,2.43)node{\footnotesize $\sigma C_{a_1}$};
 \draw (1.45,1.4)node{ $\sigma C_{a_1+2i}$};
 \draw (5,1.4)node{ $\sigma D_{a_2+2i}$};
 \draw (3,2.43)node{\footnotesize $\sigma D_{a_2}$};
 \draw (3.5,0.25)node{$\sigma C_r$};
 \draw (4.5,0.25)node{$\sigma D_s$};
 \draw (-5,0)node[left]{$P_0$};
 \draw (-5.5,-0.866)node[left]{$P$};
 \draw (4,0)node[below]{$\sigma P_0$};
 \draw (5,0)node[below]{$\sigma P$};
 \draw[->, semithick] (2.1,1.4)--(2.5,1.4);
 \draw[->, semithick] (4.3,1.4)--(3.5,1.4);
\end{tikzpicture}
\end{center}
Recall that $C_{a_1+2i}$ (resp.\ $D_{a_2+2i}$) satisfies 
\begin{align*}
\text{$\sigma C_{a_1+2i}\subset cl(C_{2i}, \sigma C_{a_1+2i-1})$ for $1\le i<m_{\mu'}$}\\
\text{(resp.\ $\sigma D_{a_2+2i}\subset cl(D_{2i}, \sigma D_{a_2+2i-1})$ for $1\le i<m_{\mu}$)}
\end{align*}
 if and only if $\rho_1(\sigma C_{a_1+2i})$ (resp.\ $\rho_2(\sigma D_{a_2+2i})$) is on the same side as $C_r$ (resp.\ $D_s$) of the wall containing $\rho_1(\sigma C_{a_1+2i}\cap \sigma C_{a_1+2i-1})$ (resp.\ $\rho_2(\sigma D_{a_2+2i}\cap \sigma D_{a_2+2i-1})$) for $1\le i<m_{\mu'}$ (resp.\ $1\le i<m_{\mu}$).
Retracting minimal galleries $(C_{2i}=D_{2i},\ldots, \sigma C_{a_1+2i})$ and $(\sigma C_{a_1+2i}, \ldots, \sigma D_{a_2+2i})$ successively by $\rho_2$,
we can also check that $\sigma C_{a_1+2i}\subset cl(C_{2i}, \sigma C_{a_1+2i-1})$ for $1\le i<m_{\mu'}$ if and only if $\rho_2(\sigma D_{a_2+2i})$ is on the same side as $D_s$ of the wall containing $\rho_2(\sigma D_{a_2+2i}\cap \sigma D_{a_2+2i-1})$ for $1\le i<m_{\mu'}=m_{\mu}$ (see Lemma \ref{LEMMA1.2} and \cite[p.\ 340]{Kottwitz}).
This implies that the inverse image of $\mathcal{G}r_{[\mu']}^{\mathrm{I\hspace{-.1em}I},m_{\mu'}}$ by the projection $\mathcal{G}r_{[\mu]}^{\mathrm{I\hspace{-.1em}I}}\rightarrow \mathcal{G}r_{[\mu']}^{\mathrm{I\hspace{-.1em}I}}$ is $\mathcal{G}r_{[\mu]}^{\mathrm{I\hspace{-.1em}I},m_{\mu}}$.
Thus, by the proof of Lemma \ref{affinebundle}, this is equivalent to saying that we have an isomorphism $$\mathcal{G}r_{[\mu]}^{\mathrm{I\hspace{-.1em}I},m_{\mu}}\cong \mathcal{G}r_{[\mu']}^{\mathrm{I\hspace{-.1em}I},m_{\mu'}}\times \mathbb A^{2(e_1+e_3-2e_2)}$$
as a locally closed subvariety of $\mathcal{G}r_{[\mu]}$.
So it is enough to consider the case $e_1-e_2=m_{\mu}$.

In the sequel, we assume that $e_1-e_2=m_{\mu}$.
In this case, we have a morphism of $\bar k$-spaces
$$\mathcal{G}r_{[\mu']}^{\mathrm{I\hspace{-.1em}I},m_{\mu'}}\rightarrow X^S,$$
which actually factors through $\mathcal{G}r_{[\lambda]}$, where $\mu'=(e_1-1,e_2,e_3), \lambda=(e_1-e_3-1,-m_{\mu'},e_3-e_1+1+m_{\mu'})$.
Indeed, we have a morphism of $\bar k$-spaces
$$J^{\mu'}\rightarrow X^S,$$
given on $R$-valued points by sending $j\in J^{\mu'}(R)$ to the lattice $t^{-\mu'}j^{-1}\sigma(j)t^{\mu'}\Lambda_{R}\in X^S(R)$.
By Proposition \ref{invformula}, the composition $\mathcal{G}r_{[\mu']}^{\mathrm{I\hspace{-.1em}I},m_{\mu'}}\subset J^{\mu'}\rightarrow X^S$ actually factors through $\mathcal{G}r_{[\lambda]}$.
Moreover, the composition $\varphi$ of the morphism $\mathcal{G}r_{[\mu']}^{\mathrm{I\hspace{-.1em}I},m_{\mu'}}\rightarrow \mathcal{G}r_{[\lambda]}$ and the canonical projection $\mathcal{G}r_{[\lambda]}\rightarrow \mathcal{G}r_{[0,0,-1]}\cong \mathbb P^2$ is a morphism of varieties factoring through a locally closed immersion $\mathbb A^1\subset \mathbb P^2$.
To check this, let $P_0\in \mathcal{G}r_{[\mu']}^{\mathrm{I\hspace{-.1em}I},m_{\mu'}}$.
Then there exists a unique matrix $g_0\in J^{\mu'}(\bar k)$ such that $P_0=[g_0t^{\mu'}\Lambda_{\bar k}]$.
Let $(C_0,\ldots, C_r)$ be the special gallery connecting $[\Lambda_{\bar k}]$ to $P_0$, and let $\mathcal A_0$ be an apartment containing $C_r$ and $\sigma C_r$.
Then the image $\varphi(P_0)$ corresponds to a vertex in the apartment $t^{-\mu'}g_{0}^{-1}\mathcal A_0$, which contains $[\mathcal O\oplus \mathcal O\oplus t\mathcal O]$ and differs from $[\mathcal O\oplus t\mathcal O\oplus t\mathcal O]$.
So $\varphi$ is actually a regular function on $\mathcal{G}r_{[\mu']}^{\mathrm{I\hspace{-.1em}I},m_{\mu'}}$.
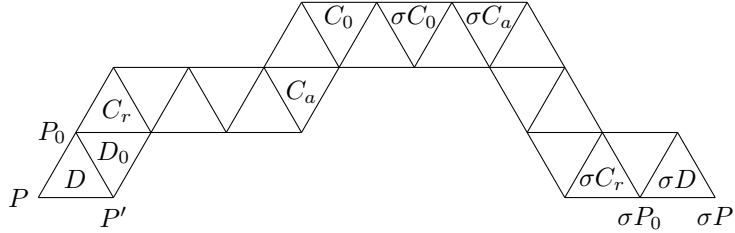
\begin{center}
\begin{tikzpicture}
 \draw (2,0)--(3,0)--cycle;
 \draw (-4.5,0.866)--(-3.5,0.866)--(-2.5,0.866)--(-1.5,0.866)--cycle;
 \draw (1.5,0.866)--(2.5,0.866)--cycle;
 \draw (2,0)--(1.5,0.866)--cycle;
 \draw (2,0)--(2.5,0.866)--cycle;
 \draw (3,0)--(2.5,0.866)--cycle;
 \draw (2.5,0.866)--(3.5,0.866)--cycle;
 \draw (3,0)--(3.5,0.866)--cycle;
 \draw (4,0)--(3.5,0.866)--cycle;
 \draw (3,0)--(4,0)--cycle;
 \draw (-4,1.732)--(-3,1.732)--(-2,1.732)--(-1,1.732)--(0,1.732)--(1,1.732)--(2,1.732)--cycle;
 \draw (-5,0)--(-4.5,0.866)--cycle;
 \draw (-5,0)--(-4,0)--cycle;
 \draw (-4,0)--(-3.5,0.866)--cycle;
 \draw (-4,0)--(-4.5,0.866)--cycle;
 \draw (-4,1.732)--(-4.5,0.866)--cycle;
 \draw (-4,1.732)--(-3.5,0.866)--cycle;
 \draw (-3,1.732)--(-3.5,0.866)--cycle;
 \draw (-3,1.732)--(-2.5,0.866)--cycle;
 \draw (-2,1.732)--(-2.5,0.866)--cycle;
 \draw (-2,1.732)--(-1.5,0.866)--cycle;
 \draw (-1,1.732)--(-1.5,0.866)--cycle;
 \draw (1,1.732)--(1.5,0.866)--cycle;
 \draw (2,1.732)--(1.5,0.866)--cycle;
 \draw (2,1.732)--(2.5,0.866)--cycle;
 \draw (-1.5,2.598)--(-0.5,2.598)--(1.5,2.598)--cycle;
 \draw (-2,1.732)--(-1.5,2.598)--cycle;
 \draw (-1,1.732)--(-1.5,2.598)--cycle;
 \draw (-1,1.732)--(-0.5,2.598)--cycle;
 \draw (0,1.732)--(-0.5,2.598)--cycle;
 \draw (0,1.732)--(0.5,2.598)--cycle;
 \draw (1,1.732)--(0.5,2.598)--cycle;
 \draw (1,1.732)--(1.5,2.598)--cycle;
 \draw (2,1.732)--(1.5,2.598)--cycle;
 \draw (-1,2.4)node{$C_0$};
 \draw (-1.5,1.4)node{$C_a$};
 \draw (-3.95,1.15)node{$C_r$};
 \draw (-4.51,0.25)node{$D$};
 \draw (-4,0.6)node{$D_0$};
 \draw (-4,0)node[below]{$P'$};
 \draw (0,2.4)node{$\sigma C_0$};
 \draw (1,2.4)node{$\sigma C_a$};
 \draw (2.5,0.25)node{$\sigma C_r$};
 \draw (3.5,0.25)node{$\sigma D$};
 \draw (-5,0)node[left]{$P$};
 \draw (4,0)node[below]{$\sigma P$};
 \draw (-4.5,0.866)node[left]{$P_0$};
 \draw (3,0)node[below]{$\sigma P_0$};
\end{tikzpicture}
\end{center}

We have a canonical projection $$\phi:\mathcal{G}r_{[\mu]}^{\mathrm{I\hspace{-.1em}I}}\rightarrow \mathcal{G}r_{[\mu']}^{\mathrm{I\hspace{-.1em}I}},$$ where $[\mu']=[e_1-1,e_2,e_3]$.
Using Lemma \ref{LEMMA1.2}, we can check that $\phi(\mathcal{G}r_{[\mu]}^{\mathrm{I\hspace{-.1em}I},m_{\mu}})\subset \mathcal{G}r_{[\mu']}^{\mathrm{I\hspace{-.1em}I},m_{\mu'}}$ (cf.\ \cite[p.\ 340]{Kottwitz}).
Moreover, by the proof of Lemma \ref{affinebundle}, we have $$\phi^{-1}(J^{\mu'})=J^{\mu}\cong J^{\mu'}\times U_{12, e_1-e_2-1}\times U_{13, e_1-e_3-1}.$$
This is equivalent to saying that any matrix in $J^{\mu}(\bar k)$ can be written as a product of a matrix in $J^{\mu'}(\bar k)$ and a matrix of the form
$$\begin{pmatrix}
1 & u_{12}t^{e_1-e_2-1} & u_{13}t^{e_1-e_3-1} \\
0 & 1 & 0 \\
0 & 0 & 1 \\
\end{pmatrix},$$
where $u_{12}, u_{13}\in \bar k$.
Let $P=[gt^{\mu}\Lambda_{\bar k}]\in \mathcal{G}r_{[\mu]}^{\mathrm{I\hspace{-.1em}I}}$ with $g\in J^{\mu}(\bar k)$, and let $D$ be the first chamber $cl(P,[\Lambda_{\bar k}])$.
Further, we assume that $P_0=[g_0t^{\mu'}\Lambda_{\bar k}]\in \mathcal{G}r_{[\mu']}^{\mathrm{I\hspace{-.1em}I},m_{\mu'}}$ with $g_0\in J^{\mu'}(\bar k)$ belongs to $cl([\Lambda_{\bar k}],P)$.
Let $(C_0,\ldots, C_r)$ be the special gallery connecting $[\Lambda_{\bar k}]$ to $P_0$, and let $(C_r, D_0, D)$ be the unique minimal gallery.
Fix an apartment $\mathcal A$ containing both $(\sigma C_0,\ldots, \sigma C_r)$ and $\sigma D$.
Then, using Lemma \ref{LEMMA1.2}, we can check that $P$ is contained in $\mathcal{G}r_{[\mu]}^{\mathrm{I\hspace{-.1em}I},m_{\mu}}$ if and only if $\rho_{\mathcal A, \sigma D}(D_0)$ is on the same side as $\sigma D$ of the wall containing $\rho_{\mathcal A, \sigma D}(C_r\cap D_0)$.
We can also check that this is equivalent to saying that $D_0\subset cl(C_r, \sigma C_r)$.
Set $D_0=\{P_0, P, P'\}$.
Then we have $P'=[g'\Lambda_{\bar k}]$, where
$$\text{$g'=g_0t^{\mu'}
\begin{pmatrix}
t & \varphi(P_0) & 0 \\
0 & 1 & 0 \\
0 & 0 & t \\
\end{pmatrix}$.}$$
This implies easily that $P$ is contained in $\mathcal{G}r_{[\mu]}^{\mathrm{I\hspace{-.1em}I},m_{\mu}}$ if and only if $g$ can be written as 
$$g=g_0t^{\mu'}
\begin{pmatrix}
t & \varphi(P_0) & u_{13} \\
0 & 1 & 0 \\
0 & 0 & 1 \\
\end{pmatrix}t^{-\mu}
=g_0
\begin{pmatrix}
1 & \varphi(P_0)t^{e_1-e_2-1} & u_{13}t^{e_1-e_3-1} \\
0 & 1 & 0 \\
0 & 0 & 1 \\
\end{pmatrix},$$
where $g\in \mathcal{G}r_{[\mu']}^{\mathrm{I\hspace{-.1em}I},m_{\mu'}}$ and $\varphi(P_0), u_{13}\in \bar k$.
Thus $\mathcal{G}r_{[\mu]}^{\mathrm{I\hspace{-.1em}I},m_{\mu}}$ is a closed subvariety of $\phi^{-1}(\mathcal{G}r_{[\mu']}^{\mathrm{I\hspace{-.1em}I},m_{\mu'}})\cong\mathcal{G}r_{[\mu']}^{\mathrm{I\hspace{-.1em}I},m_{\mu'}}\times U_{12, e_1-e_2-1}\times U_{13, e_1-e_3-1}$ defined by the equation $u_{12}=\varphi$.
By the induction hypothesis, this is isomorphic to 
$$\mathcal{G}r_{[\mu']}^{\mathrm{I\hspace{-.1em}I},m_{\mu'}}\times \mathbb A^1\cong \Omega\times\mathbb A^{2(e_1-e_3)-m_{\mu'}-2}\times \mathbb A^1\cong \Omega\times \mathbb A^{2(e_1-e_3)-m_{\mu}-2}.$$

Finally, we show the statement for $\mathcal{G}r_{[\mu]}^{\mathrm{I\hspace{-.1em}I},j}$.
Set 
$$G_{[\mu]}^{\mathrm{I\hspace{-.1em}I},j}=
\{P\in \mathcal{G}r_{[\mu]}^{\mathrm{I\hspace{-.1em}I}}\mid 
\text{$\sigma C_{a+2i}\subset cl(C_{2i}, \sigma C_{a+2i-1})$ for $1\le i< j$}\}.
$$
Then the similar proof as above shows that $G_{[\mu]}^{\mathrm{I\hspace{-.1em}I},j}$ is locally closed and its reduced subscheme structure is isomorphic to $\Omega\times \mathbb A^{2(e_1-e_3)-j-2}$.
Again, the similar argument as above shows that $\mathcal{G}r_{[\mu]}^{\mathrm{I\hspace{-.1em}I},j}$ is an open subvariety of $G_{[\mu]}^{\mathrm{I\hspace{-.1em}I},j}$ isomorphic to $\Omega\times \mathbb G_m\times \mathbb A^{2(e_1-e_3)-j-3}$, which completes the proof.
\end{proof}

We keep the notation in the proof of Lemma \ref{invfinite2}, and introduce the subsets of $\mathcal{G}r_{[\mu]}(Q)$ for any vertex $Q$ in $C_M$.
For any $[\mu]=[e_1,e_2,e_3]\in X_*(T)_+$, set
\begin{equation*}
\mathcal{G}r_{[\mu],C_M}^{u}(Q)=
\left \{P\in \mathcal{G}r_{[\mu]}(Q) \left|
\begin{array}{l}
\text{there exists a first chamber $C$ of $cl(Q,P)$}\\
\text{such that $C_M$ and $C$ have relative position $u$}
\end{array}
\right.\right\},
\end{equation*}
where $u=$ I\hspace{-.1em}I or I\hspace{-.1em}I\hspace{-.1em}I.
In case $e_1=e_2$ or $e_2=e_3$, then $P\in \mathcal{G}r_{[\mu]}(Q)$ belongs to $\mathcal{G}r_{[\mu]}^{u}(Q)$ if and only if the first ``vertex" of $cl(Q,P)$ is not contained in $C_M$.
In the case where $C_M$ and $C_0$ have relative position I, we will define two subvarieties.
For any $[\mu]=[e_1,e_2,e_3] \in X_*(T)_+$ with $e_1>e_2>e_3$, set
\begin{equation*}
\mathcal{G}r_{[\mu],D}^{u}(Q)=
\left \{P\in \mathcal{G}r_{[\mu]}(Q) \left|
\begin{array}{l}
\text{$C_0$ and $b_1\sigma D$ have relative position $u$, where}\\
\text{$C_0$ is the unique first chamber of $cl(Q,P)$}\\
\end{array}
\right.\right\},
\end{equation*}
where $u=$ I or I\hspace{-.1em}I, and similarly, set
\begin{equation*}
\mathcal{G}r_{[\mu],D'}^{u}(Q)=
\left \{P\in \mathcal{G}r_{[\mu]}(Q) \left|
\begin{array}{l}
\text{$C_0$ and $b_2\sigma D'$ have relative position $u$, where}\\
\text{$C_0$ is the unique first chamber of $cl(Q,P)$}\\
\end{array}
\right.\right\},
\end{equation*}
where $u=$ I or I\hspace{-.1em}I\hspace{-.1em}I.

Let $a=2(e_1-e_2-1)$.
For any $[\mu]=[e_1,e_2,e_3] \in X_*(T)_+$ with $e_1>e_2>e_3$, we also define the sets
\begin{equation*}
\mathcal{G}r_{[\mu],b_1}^{\mathrm{I\hspace{-.1em}I},j}(Q)=
\left \{P\in \mathcal{G}r_{[\mu],C_M}^{\mathrm{I\hspace{-.1em}I}}(Q) \left|
\begin{array}{l}
\text{$C_{a+2i}\subset cl(b_1\sigma C_{2(i-1)}, C_{a+2i-1})$, $1\le i<j$}\\
\text{and $C_{a+2j}\nsubseteq cl(b_1\sigma C_{2(j-1)}, C_{a+2j-1})$}\\
\end{array}
\right.\right\},
\end{equation*}
where $1\le j<m_{\mu, \mathrm{I\hspace{-.1em}I}}$, and
$$
\mathcal{G}r_{[\mu],b_1}^{\mathrm{I\hspace{-.1em}I},j}(Q)=
\{P\in \mathcal{G}r_{[\mu],C_M}^{\mathrm{I\hspace{-.1em}I}}(Q)\mid 
\text{$C_{a+2i}\subset cl(b_1\sigma C_{2(i-1)}, C_{a+2i-1})$, $1\le i< m_{\mu,\mathrm{I\hspace{-.1em}I}}$}\},
$$
where $j=m_{\mu, \mathrm{I\hspace{-.1em}I}}$.
These conditions appear when we connect the two galleries $(C_0, \ldots, C_r)$ and $(C_M,\ldots, b_1\sigma C_0,\ldots, b_1\sigma C_r)$.
In the case where $C_M$ and $C_0$ have relative position I, we define
\begin{equation*}
\mathcal{G}r_{[\mu],D}^{\mathrm{I\hspace{-.1em}I},j}(Q)=
\left \{P\in \mathcal{G}r_{[\mu],D}^{\mathrm{I\hspace{-.1em}I}}(Q) \left|
\begin{array}{l}
\text{$C_{2i}\subset cl(b_1\sigma C_{a+2i}, C_{2i-1})$, $1\le i<j$}\\
\text{and $C_{2j}\nsubseteq cl(b\sigma C_{a+2j}, C_{2j-1})$}\\
\end{array}
\right.\right\},
\end{equation*}
where $1\le j< m_{\mu,\mathrm{I}}$, and
$$
\mathcal{G}r_{[\mu],D}^{\mathrm{I\hspace{-.1em}I},j}(Q)=
\{P\in \mathcal{G}r_{[\mu],D}^{\mathrm{I\hspace{-.1em}I}}(Q)\mid 
\text{$C_{2i}\subset cl(b_1\sigma C_{a+2i}, C_{2i-1})$, $1\le i< m_{\mu,\mathrm{I}}$}\},
$$
where $j=m_{\mu,\mathrm{I}}$.
These conditions appear when we connect the two galleries $(C_0, \ldots, C_r)$ and $(b_1\sigma D, \ldots, b_1\sigma C_0, \ldots, b_1\sigma C_r)$.
Similarly, by writing down the conditions on a bend of galleries appearing in the proof of Lemma \ref{invfinite2}, we define 
\begin{align*}
&\mathcal{G}r_{[\mu],b_2}^{\mathrm{I\hspace{-.1em}I},j}(Q)\quad (1\le j\le m_{\mu,\mathrm{I\hspace{-.1em}I}}),\\
&\mathcal{G}r_{[\mu],b_1}^{\mathrm{I\hspace{-.1em}I\hspace{-.1em}I},j}(Q), \quad \mathcal{G}r_{[\mu],b_2}^{\mathrm{I\hspace{-.1em}I\hspace{-.1em}I},j}(Q)\quad (1\le j\le m_{\mu,\mathrm{I\hspace{-.1em}I\hspace{-.1em}I}}),\\
&\mathcal{G}r_{[\mu],D'}^{\mathrm{I\hspace{-.1em}I\hspace{-.1em}I},j}(Q)\quad (1\le j\le m_{\mu,\mathrm{I}}).
\end{align*}
The subsets of $\mathcal{G}r_{[\mu]}(Q)$ defined above can be seen as the locally closed reduced subvarieties of the variety $\mathcal{G}r_{[\mu]}(Q)$, and denoted by the same symbols (see the next proposition).

\begin{prop}
\label{geometricstructure2}
Let notation be as above.
Let $b=b_1$ or $b_2$.
\begin{enumerate}[(i)]
\item Let $[\mu]=[e_1,e_2,e_3] \in X_*(T)_+$ with $e_1=e_2>e_3$.
Then the $\bar k$-variety $\mathcal{G}r_{[\mu],C_M}^{\mathrm{I\hspace{-.1em}I}}(Q)$ (resp.\ $\mathcal{G}r_{[\mu],C_M}^{\mathrm{I\hspace{-.1em}I\hspace{-.1em}I}}(Q)$) is isomorphic to $\mathbb A^{2(e_1-e_3)-1}$ (resp.\ $\mathbb A^{2(e_1-e_3)}$).

\item Let $[\mu]=[e_1,e_2,e_3] \in X_*(T)_+$ with $e_1>e_2=e_3$.
Then the $\bar k$-variety $\mathcal{G}r_{[\mu],C_M}^{\mathrm{I\hspace{-.1em}I}}(Q)$ (resp.\ $\mathcal{G}r_{[\mu],C_M}^{\mathrm{I\hspace{-.1em}I\hspace{-.1em}I}}(Q)$) is isomorphic to $\mathbb A^{2(e_1-e_3)}$ (resp.\ $\mathbb A^{2(e_1-e_3)-1}$).

\item Let $[\mu]=[e_1,e_2,e_3] \in X_*(T)_+$ with $e_1>e_2>e_3$.
Then the $\bar k$-variety $\mathcal{G}r_{[\mu],C_M}^{\mathrm{I\hspace{-.1em}I}}(Q)$ (resp.\ $\mathcal{G}r_{[\mu],b}^{\mathrm{I\hspace{-.1em}I},j}(Q)$, resp.\ $\mathcal{G}r_{[\mu],b}^{\mathrm{I\hspace{-.1em}I}, m_{\mu,\mathrm{I\hspace{-.1em}I}}}(Q)$) is isomorphic to $\mathbb A^{2(e_1-e_3)-1}$ (resp.\ $\mathbb G_m\times\mathbb A^{2(e_1-e_3)-j-1}$, resp.\ $\mathbb A^{2(e_1-e_3)-m_{\mu,\mathrm{I\hspace{-.1em}I}}}$).

\item Let $[\mu]=[e_1,e_2,e_3] \in X_*(T)_+$ with $e_1>e_2>e_3$.
Then the $\bar k$-variety $\mathcal{G}r_{[\mu],C_M}^{\mathrm{I\hspace{-.1em}I\hspace{-.1em}I}}(Q)$ (resp.\ $\mathcal{G}r_{[\mu],b}^{\mathrm{I\hspace{-.1em}I\hspace{-.1em}I},j}(Q)$, resp.\ $\mathcal{G}r_{[\mu],b}^{\mathrm{I\hspace{-.1em}I\hspace{-.1em}I}, m_{\mu,\mathrm{I\hspace{-.1em}I\hspace{-.1em}I}}}(Q)$) is isomorphic to $\mathbb A^{2(e_1-e_3)-1}$ (resp.\ $\mathbb G_m\times\mathbb A^{2(e_1-e_3)-j-1}$, resp.\ $\mathbb A^{2(e_1-e_3)-m_{\mu,\mathrm{I\hspace{-.1em}I\hspace{-.1em}I}}}$).

\item Let $[\mu]=[e_1,e_2,e_3] \in X_*(T)_+$ with $e_1>e_2>e_3$.
Then the $\bar k$-variety $\mathcal{G}r_{[\mu],D}^{\mathrm{I}}(Q)$ (resp.\ $\mathcal{G}r_{[\mu],D}^{\mathrm{I\hspace{-.1em}I}}(Q)$, resp.\ $\mathcal{G}r_{[\mu],D}^{\mathrm{I\hspace{-.1em}I},j}(Q)$, resp.\ $\mathcal{G}r_{[\mu],D}^{\mathrm{I\hspace{-.1em}I},m_{\mu,\mathrm{I}}}(Q)$) is isomorphic to $\mathbb G_m\times \mathbb A^{2(e_1-e_3)-1}$ (resp.\ $\mathbb A^{2(e_1-e_3)-1}$, resp.\ $\mathbb G_m\times\mathbb A^{2(e_1-e_3)-j-1}$, resp.\ $\mathbb A^{2(e_1-e_3)-m_{\mu,\mathrm{I}}}$).

\item Let $[\mu]=[e_1,e_2,e_3] \in X_*(T)_+$ with $e_1>e_2>e_3$.
Then the $\bar k$-variety $\mathcal{G}r_{[\mu],D'}^{\mathrm{I}}(Q)$ (resp.\ $\mathcal{G}r_{[\mu],D'}^{\mathrm{I\hspace{-.1em}I\hspace{-.1em}I}}(Q)$, resp.\ $\mathcal{G}r_{[\mu],D'}^{\mathrm{I\hspace{-.1em}I\hspace{-.1em}I},j}(Q)$, resp.\ $\mathcal{G}r_{[\mu],D'}^{\mathrm{I\hspace{-.1em}I\hspace{-.1em}I},m_{\mu,\mathrm{I}}}(Q)$) is isomorphic to $\mathbb G_m\times \mathbb A^{2(e_1-e_3)-1}$ (resp.\ $\mathbb A^{2(e_1-e_3)-1}$, resp.\ $\mathbb G_m\times\mathbb A^{2(e_1-e_3)-j-1}$, resp.\ $\mathbb A^{2(e_1-e_3)-m_{\mu,\mathrm{I}}}$).
\end{enumerate}
\end{prop}
\begin{proof}
It suffices to prove the case for $Q=[\Lambda_{\bar k}]$.
We omit $[\Lambda_{\bar k}]$ from the notation.
Define
$$G_1=\{P\in \mathcal{G}r_{[1,0,0]}\mid \text{$\{[\Lambda_{\bar k}], P, [t\mathcal O\oplus t\mathcal O\oplus \mathcal O]\}$ is not a chamber}\}$$
and
\begin{equation*}
G_1'=
\left \{P\in \mathcal{G}r_{[1,0,0]} \left|
\begin{array}{l}
\text{$P\neq [t\mathcal O\oplus \mathcal O\oplus \mathcal O]$ and }\\
\text{$\{[\Lambda_{\bar k}], P, [t\mathcal O\oplus t\mathcal O\oplus \mathcal O]\}$ is a chamber}\\
\end{array}
\right.\right\}.
\end{equation*}
Then both $G_1$ and $G_1'$ are locally closed subsets, and we have 
$$G_1\cong \mathbb A^2,\quad G_1'\cong \mathbb A^1$$
as reduced $\bar k$-varieties.
These isomorphisms and the proof of Lemma \ref{affinebundle} imply (ii), and the proof for (i) is similar.

Let $C$ be a chamber containing $[\Lambda_{\bar k}]$ such that $C_M$ and $C$ have relative position I\hspace{-.1em}I.
Then $C$ is completely determined by the vertex $P_1$ of type $1$ in $C$.
Indeed, the vertex $P_2$ of type $2$ in $C$ belongs to $cl(C_M, P_1)$.
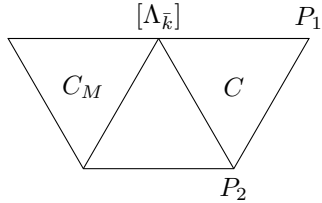
\begin{center}
\begin{tikzpicture}
 \draw (-2,0)--(2,0)--(1,-1.732)--(-1,-1.732)--cycle;
 \draw (0,0)--(-2,0)--cycle;
 \draw (0,0)--(2,0)--cycle;
 \draw (0,0)--(1,-1.732)--cycle;
 \draw (0,0)--(-1,-1.732)--cycle;
 \draw (0,0)node[above]{$[\Lambda_{\bar k}]$};
 \draw (1,-1.732)node[below]{$P_2$};
 \draw (2,0)node[above]{$P_1$};
 \draw (-1,-0.65)node{$C_M$};
 \draw (1,-0.65)node{$C$};
\end{tikzpicture}
\end{center}
Therefore, the set of such chambers (which can be identified with a subset of $\mathcal{G}r_1(\bar k)\times \mathcal{G}r_{-1}(\bar k)$) is isomorphic to $G_1$ and the proof of Lemma \ref{affinebundle} implies the statement for $\mathcal{G}r_{[\mu],C_M}^{\mathrm{I\hspace{-.1em}I}}$.
Moreover, by an argument similar to Proposition \ref{geometricstructure}, we can show the statements for $\mathcal{G}r_{[\mu],b}^{\mathrm{I\hspace{-.1em}I},j}$ and $\mathcal{G}r_{[\mu],b}^{m_{\mu,\mathrm{I\hspace{-.1em}I}}}$.
So we obtain (iii), and the proof for (iv) is the same.

Finally, we show (v), and (vi) follows similarly.
Define
\begin{equation*}
F_{\mathrm{I}}=
\left \{C_0\in \mathop{\mathrm{Flag}}(\bar k) \left|
\begin{array}{l}
\text{$C_M$ and $C_0$ have relative position I and }\\
\text{$C_0$ and $b_1\sigma D$ have relative position I}\\
\end{array}
\right.\right\}
\end{equation*}
and
\begin{equation*}
F_{\mathrm{I\hspace{-.1em}I}}=
\left \{C_0\in \mathop{\mathrm{Flag}}(\bar k) \left|
\begin{array}{l}
\text{$C_M$ and $C_0$ have relative position I and }\\
\text{$C_0$ and $b_1\sigma D$ have relative position I\hspace{-.1em}I}\\
\end{array}
\right.\right\}.
\end{equation*}
Using the techniques explained above, we can show that
$$F_{\mathrm{I}}\cong \mathbb G_m\times \mathbb A^2,\quad F_{\mathrm{I\hspace{-.1em}I}}\cong \mathbb A^2$$
as reduced $\bar k$-varieties.
These isomorphisms show the assertion in the same way as above.
\end{proof}

\section{Geometric Structure of Affine Deligne-Lusztig Varieties}
Let $\lambda\in X_*(T)_+$ with $X_{\lambda}(1)\neq \emptyset$.
Recall that $\mathcal P_{\lambda}(1)$ can be seen as the set
\begin{equation*}
\left \{(Q, [\mu]) \left|
\begin{array}{l}
\text{$[\mu]=[e_1, e_2, e_3]\in M_{\lambda}(1)$ and}\\
\text{$Q$ is a vertex of type $-(e_1+e_2+e_3)\in \mathbb Z/3$}\\
\end{array}
\right.\right\}
\end{equation*}
(Section \ref{b=1}).
For any such $\lambda$, we have explicit description of $M_{\lambda}(1)$ (Lemma \ref{invfinite}) and $H_1$ acts on $\mathcal P_{\lambda}(1)$ by left multiplication on the vertex $Q$.
Set $K_0=\mathrm{GL}_3(\mathcal O_F)$, $\Lambda_0=\Lambda_k$ and 
\begin{align*}
K_1 &=
\begin{pmatrix}
t & 0 & 0 \\
0 & 1 & 0 \\
0 & 0 & 1 \\
\end{pmatrix}
\mathrm{GL}_3(\mathcal O_F)
\begin{pmatrix}
t^{-1} & 0 & 0 \\
0 & 1 & 0 \\
0 & 0 & 1 \\
\end{pmatrix},\quad \Lambda_1=t\mathcal O\oplus \mathcal O\oplus \mathcal O,\\
K_2 &=
\begin{pmatrix}
t & 0 & 0 \\
0 & t & 0 \\
0 & 0 & 1 \\
\end{pmatrix}
\mathrm{GL}_3(\mathcal O_F)
\begin{pmatrix}
t^{-1} & 0 & 0 \\
0 & t^{-1} & 0 \\
0 & 0 & 1 \\
\end{pmatrix},\quad \Lambda_2=t\mathcal O\oplus t\mathcal O\oplus \mathcal O.
\end{align*}
Let $[\mu]=[e_1, e_2, e_3]\in M_{\lambda}(1)$, and let $i\in \{0, 1, 2\}$ be the representative of $-(e_1+e_2+e_3)\in \mathbb Z/3$.
Note that these groups are subgroups of $H_1\subset J_1=\mathrm{GL}_3(F)$, and $K_i$ is the stabilizer of $([\Lambda_i], [\mu])$ with respect to the action of $H_1$ on $\mathcal{P}_{\lambda}(1)$.
Finally, we define $K_{[\mu]}$ as $K_i$ and set $m_{\mu}=\min\{e_1-e_2, e_2-e_3\}$.

\begin{theo}
\label{ADLVgeometricstructure}
Let notation be as above.
Then the irreducible components of $X_{\lambda}(1)$ are parameterized by the elements in $\bigsqcup_{[\mu]\in M_{\lambda}(1)}J_1/K_{[\mu]}$, and $J_1$ acts on the set of irreducible components by left multiplication on this set.
Moreover, their geometric structures of $\bar k$-varieties are given as follows:
\begin{enumerate}[(i)]
\item Let $\lambda=(m_1, 0, m_3)\in X_*(T)_+$ with $X_{\lambda}(1)\neq \emptyset$.
For any $[\mu]=[e_1, e_2, e_3]\in M_{\lambda}(1)$ with $e_1=e_2>e_3$ or $e_1>e_2=e_3$ (resp.\ $e_1>e_2>e_3$), the irreducible component corresponding to $gK_{[\mu]}\in J_1/K_{[\mu]}$ is an affine bundle of rank $m_1-m_3-2$ (resp.\ $m_1-m_3-3$) over $\mathbb P^2\setminus \mathbb P^2(k)$ (resp.\ $X_{\mathrm{I}}$).
If $[\mu]=[0,0,0]\in M_{\lambda}(1)$ (and hence $\lambda=(0,0,0)$), then the irreducible component corresponding to $gK_{[\mu]}\in J_1/K_{[\mu]}$ is a point.

\item Let $\lambda=(m_1, m_2, m_3)\in X_*(T)_+, m_2<0$ with $X_{\lambda}(1)\neq \emptyset$.
For any $[\mu]=[e_1, e_2, e_3]\in M_{\lambda}(1)$ with $m_{\mu}=-m_2$ (resp.\  $m_{\mu}>-m_2$),
the irreducible component corresponding to $gK_{[\mu]}\in J_1/K_{[\mu]}$ is isomorphic to 
$$\text{$\Omega \times \mathbb A^{m_1-m_3-2}$ (resp.\ $\Omega\times \mathbb G_m\times \mathbb A^{m_1-m_3-3}$)}.$$
In this case, the irreducible components are pairwise disjoint.

\item Let $\lambda=(m_1, m_2, m_3)\in X_*(T)_+, m_2>0$ with $X_{\lambda}(1)\neq \emptyset$.
For any $[\mu]=[e_1, e_2, e_3]\in M_{\lambda}(1)$ with $m_{\mu}=m_2$ (resp.\  $m_{\mu}>m_2$),
the irreducible component corresponding to $gK_{[\mu]}\in J_1/K_{[\mu]}$ is isomorphic to 
$$\text{$\Omega \times \mathbb A^{m_1-m_3-2}$ (resp.\ $\Omega\times \mathbb G_m\times \mathbb A^{m_1-m_3-3}$)}.$$
In this case, the irreducible components are pairwise disjoint.
\end{enumerate}
\end{theo}

\begin{proof}
Fix $\lambda=(m_1, m_2, m_3)\in X_*(T)_+$ with $X_{\lambda}(1)\neq \emptyset$.
The case where $\lambda=(0,0,0)$ is well-known.
So we may assume $\lambda\neq (0,0,0)$.
First note that we have a decomposition
$$X_{\lambda}^S(1)(\bar k)=\bigcup_{(Q,[\mu])\in\mathcal P_{\lambda}(1)}(X_{\lambda}^S(1)(\bar k)\cap G_{[\mu]}^{\mathcal B_1}(Q))$$
by Proposition \ref{invdecomposition}.
We show that for any $(Q,[\mu])\in\mathcal P_{\lambda}(1)$, the subset $X_{\lambda}^S(1)(\bar k)\cap G_{[\mu]}^{\mathcal B_1}(Q)$ is an irreducible component of $X_{\lambda}(1)(\bar k)$.
By Proposition \ref{invformula}, we have
\[
  X_{\lambda}^S(1)(\bar k)\cap G_{[\mu]}^{\mathcal B_1}(Q)=
  \begin{cases}
    \mathcal{G}r_{[\mu]}^{\mathrm{I}}(Q) & (m_2=0) \\
    \mathcal{G}r_{[\mu]}^{\mathrm{I\hspace{-.1em}I},-m_2}(Q) & (m_2<0) \\
    \mathcal{G}r_{[\mu]}^{\mathrm{I\hspace{-.1em}I\hspace{-.1em}I},m_2}(Q) & (m_2>0). \\
  \end{cases}
\]
So, by Lemma \ref{invclosed} and Proposition \ref{geometricstructure}, it is an irreducible closed subset of dimension $m_1-m_3$, which is also the dimension of $X_{\lambda}(1)$ (see for example \cite[Theorem 4.17]{Gortz}).
Thus, using Proposition \ref{ADLVdecomposition}, $X_{\lambda}^S(1)(\bar k)\cap G_{[\mu]}^{\mathcal B_1}(Q)$ is an irreducible component of $X_{\lambda}(1)(\bar k)$.

Since the action of $H_1$ on $\mathcal P_{\lambda}(1)$ induces $H_1\backslash \mathcal P_{\lambda}(1)\cong M_{\lambda}(1)$, we have
$$X_{\lambda}(1)\cong \bigsqcup_{J_1/H_1}\bigcup_{\bigsqcup_{[\mu]\in M_{\lambda}(1)} H_1/K_{[\mu]}}\mathcal{G}r_{[\mu]}^{u, j}(Q)= \bigcup_{\bigsqcup_{[\mu]\in M_{\lambda}(1)} J_1/K_{[\mu]}}\mathcal{G}r_{[\mu]}^{u, j}(Q),$$
where $J_1$ acts on the set of the irreducible components by left multiplication on the index set.
Again by Proposition \ref{geometricstructure}, the geometric structure of each irreducible component is given as above.

Let $V$ be an irreducible component of $X_{\lambda}^S(1)$.
Then $V$ is quasi-compact (cf.\ \cite[Corollary 6.5]{HV}), and hence
$$V(\bar k)\cap G_{[\mu]}^{\mathcal B_1}(Q)=\emptyset$$
for all but finitely many $(Q,[\mu])\in\mathcal P_{\lambda}(1)$.
Indeed, using \cite[Lemma 2.4]{Gortz}, we can show that the function $V(\bar k)\rightarrow \mathbb N$ which maps $P\in V(\bar k)$ to the distance from $P$ to $[\Lambda_{\bar k}]$ is bounded.
Since $V$ is an irreducible component, we have $V=X_{\lambda}^S(1)(\bar k)\cap G_{[\mu]}^{\mathcal B_1}(Q)$ for some $(Q,[\mu])\in\mathcal P_{\lambda}(1)$.
On the other hand, by the similar argument as below, we can easily show that 
$$X_{\lambda}^S(1)(\bar k)\cap G_{[\mu]}^{\mathcal B_1}(Q)\neq X_{\lambda}^S(1)(\bar k)\cap G_{[\mu']}^{\mathcal B_1}(Q')$$
unless $(Q, [\mu])=(Q', [\mu'])$.
So all of the irreducible components of $X_{\lambda}(1)$ are parameterized by the elements in $\bigsqcup_{[\mu]\in M_{\lambda}(1)}J_1/K_{[\mu]}$.

Finally, let us consider the cases (ii) and (iii).
Let $(Q, [\mu]), (Q', [\mu'])\in \mathcal P_{\lambda}(1)$, then we have $$X_{\lambda}^S(1)(\bar k)\cap G_{[\mu]}^{\mathcal B_1}(Q)\cap G_{[\mu']}^{\mathcal B_1}(Q')=\emptyset$$ unless $(Q, [\mu])=(Q', [\mu'])$.
To show this, recall that we can identify $\mathop{\mathrm{Flag}}(\bar k)$ with the set of chambers containing $[\Lambda_{\bar k}]$ (Section \ref{subvarieties}).
If $C$ is a chamber in $X_{\mathrm{I\hspace{-.1em}I}}\cup X_{\mathrm{I\hspace{-.1em}I\hspace{-.1em}I}}$ and if $D$ is a chamber in $\mathcal B_1$ containing $[\Lambda_{\bar k}]$, then $C$ and $D$ always have relative position I.
This can be checked as in the first part of the proof of Proposition \ref{geometricstructure}.
Thus, by Proposition \ref{invformula}, the vertex $Q$ is the unique nearest one in $\mathcal B_1$ to any vertex in $X_{\lambda}^S(1)(\bar k)\cap G_{[\mu]}^{\mathcal B_1}(Q)$.
This completes the proof.
\end{proof}

\begin{rema}
In the case (i) of Theorem \ref{ADLVgeometricstructure}, the irreducible components are not disjoint in general.
For example, there exists a vertex $P$ of type $0$ in $\mathcal B_{\infty}$ such that $\{P, [\Lambda_1], [\Lambda_2]\}$ is a chamber and $P$ is not a vertex in $\mathcal B_1$.
Then $P$ belongs to both $\mathcal{G}r_{[0,0,-1]}^{\mathrm{I}}([\Lambda_1])$ and $\mathcal{G}r_{[1,0,0]}^{\mathrm{I}}([\Lambda_2])$, which are irreducible components of $X_{(1,0,-1)}(1)$.
On the other hand, the disjoint decomposition in (ii) and (iii) is an example of $J$-stratification introduced by Chen and Viehmann in \cite{CV}.
\end{rema}

Next, let us consider the superbasic case.
In this case, we need some notation in addition to those in Section \ref{superbasic}.
Let $\lambda=(m_1, m_2, m_3)\in X_*(T)_+$ with $X_{\lambda}(b_i)\neq \emptyset$ ($i=1, 2$).
We define
\[
  M_{\lambda}(b_1)'=
  \begin{cases}
    M_{\lambda}(b_1) & (m_2\le0) \\
    \mathbb \{[\mu]\in M_{\lambda}(b_1)\mid m_{\mu, \mathrm{I\hspace{-.1em}I}}\geq m_2\} & (m_2>0), \\
  \end{cases}
\]
and similarly,
\[
  M_{\lambda}(b_2)'=
  \begin{cases}
    M_{\lambda}(b_2) & (m_2\geq 1) \\
    \mathbb \{[\mu]\in M_{\lambda}(b_1)\mid m_{\mu, \mathrm{I\hspace{-.1em}I\hspace{-.1em}I}}\geq -m_2+1\} & (m_2<1), \\
  \end{cases}
\]
where $m_{\mu,\mathrm{I}}=\min\{e_1-e_2,e_2-e_3\}, m_{\mu,\mathrm{I\hspace{-.1em}I}}=\min\{e_1-e_2+1,e_2-e_3\}, m_{\mu,\mathrm{I\hspace{-.1em}I\hspace{-.1em}I}}=\min\{e_1-e_2,e_2-e_3+1\}$ for any $[\mu]=[e_1, e_2, e_3]\in X_*(T)_+'$.
Finally, note that $H_{b_i}$ stabilizes $[\Lambda_0], [\Lambda_1]$ and $[\Lambda_2]$ under the action on $\mathcal B_{\infty}$.
In fact, for any $g\in H_{b_i}$, the equation $v_{L}(\det(g))=0$ implies that
$$g\in\begin{pmatrix}
\mathcal O^{\times} & (t) & (t)\\
\mathcal O & \mathcal O^{\times} & (t) \\
\mathcal O & \mathcal O & \mathcal O^{\times} \\
\end{pmatrix}.$$

\begin{theo}
\label{ADLVgeometricstructure2}
Let notation be as above.
Then the irreducible components of $X_{\lambda}(b_i)$ ($i=1, 2$) are parameterized by the elements in $(J_{b_i}/H_{b_i})\times M_{\lambda}(b_i)'$, and $J_{b_i}$ acts on the set of irreducible components by left multiplication on this set.
Moreover, their geometric structures of $\bar k$-varieties are given as follows:
\begin{enumerate}[(i)]
\item Let $\lambda=(m_1, 0, m_3)\in X_*(T)_+$ with $X_{\lambda}(b_1)\neq \emptyset$.
For any $[\mu]=[e_1, e_2, e_3]\in M_{\lambda}(b_1)$ with $e_1=e_2>e_3$ or $e_1>e_2=e_3$ (resp.\ $e_1>e_2>e_3$), the irreducible component corresponding to $(gH_{b_1},[\mu])\in (J_{b_1}/H_{b_1})\times M_{\lambda}(b_1)$ is isomorphic to $$\text{$\mathbb A^{m_1-m_3-1}$ (resp.\ $\mathbb G_m\times \mathbb A^{m_1-m_3-2}$)}.$$
If $[\mu]=[0,0,0]$ (and hence $\lambda=(1,0,0)$), then the irreducible component corresponding to $(gH_{b_1},[\mu])\in (J_{b_1}/H_{b_1})\times M_{\lambda}(b_1)$ is a point.

\item Let $\lambda=(m_1, m_2, m_3)\in X_*(T)_+, m_2<0$ with $X_{\lambda}(b_1)\neq \emptyset$.
For any $[\mu]=[e_1, e_2, e_3]\in M_{\lambda}(b_1)$ with $m_{\mu, \mathrm{I}}=-m_2$ (resp.\  $m_{\mu, \mathrm{I}}>-m_2$),
the irreducible component corresponding to $(gH_{b_1},[\mu])\in (J_{b_1}/H_{b_1})\times M_{\lambda}(b_1)$ is isomorphic to 
$$\text{$\mathbb A^{m_1-m_3-1}$ (resp.\ $\mathbb G_m\times \mathbb A^{m_1-m_3-2}$)}.$$

\item Let $\lambda=(m_1, m_2, m_3)\in X_*(T)_+, m_2>0$ with $X_{\lambda}(b_1)\neq \emptyset$.
For any $[\mu]=[e_1, e_2, e_3]\in M_{\lambda}(b_1)'$ with $m_{\mu, \mathrm{I\hspace{-.1em}I}}=m_2$ (resp.\  $m_{\mu, \mathrm{I\hspace{-.1em}I}}>m_2$),
the irreducible component corresponding to $(gH_{b_1}, [\mu])\in (J_{b_1}/H_{b_1})\times M_{\lambda}(b_1)'$ is isomorphic to 
$$\text{$\mathbb A^{m_1-m_3-1}$ (resp.\ $\mathbb G_m\times \mathbb A^{m_1-m_3-2}$)}.$$

\item Let $\lambda=(m_1, 1, m_3)\in X_*(T)_+$ with $X_{\lambda}(b_2)\neq \emptyset$.
For any $[\mu]=[e_1, e_2, e_3]\in M_{\lambda}(b_2)$ with $e_1=e_2>e_3$ or $e_1>e_2=e_3$ (resp.\ $e_1>e_2>e_3$), the irreducible component corresponding to $(gH_{b_2},[\mu])\in (J_{b_2}/H_{b_2})\times M_{\lambda}(b_2)$ is isomorphic to $$\text{$\mathbb A^{m_1-m_3-1}$ (resp.\ $\mathbb G_m\times \mathbb A^{m_1-m_3-2}$)}.$$
If $[\mu]=[0,0,0]$ (and hence $\lambda=(1,1,0)$), then the irreducible component corresponding to $(gH_{b_2},[\mu])\in (J_{b_2}/H_{b_2})\times M_{\lambda}(b_2)$ is a point.

\item Let $\lambda=(m_1, m_2, m_3)\in X_*(T)_+, m_2<1$ with $X_{\lambda}(b_2)\neq \emptyset$.
For any $[\mu]=[e_1, e_2, e_3]\in M_{\lambda}(b_2)'$ with $m_{\mu, \mathrm{I\hspace{-.1em}I\hspace{-.1em}I}}=-m_2+1$ (resp.\  $m_{\mu, \mathrm{I\hspace{-.1em}I\hspace{-.1em}I}}>-m_2+1$),
the irreducible component corresponding to $(gH_{b_2}, [\mu])\in (J_{b_2}/H_{b_2})\times M_{\lambda}(b_2)'$ is isomorphic to 
$$\text{$\mathbb A^{m_1-m_3-1}$ (resp.\ $\mathbb G_m\times \mathbb A^{m_1-m_3-2}$)}.$$

\item Let $\lambda=(m_1, m_2, m_3)\in X_*(T)_+, m_2>1$ with $X_{\lambda}(b_2)\neq \emptyset$.
For any $[\mu]=[e_1, e_2, e_3]\in M_{\lambda}(b_2)$ with $m_{\mu, \mathrm{I}}=m_2-1$ (resp.\  $m_{\mu, \mathrm{I}}>m_2-1$),
the irreducible component corresponding to $(gH_{b_2},[\mu])\in (J_{b_2}/H_{b_2})\times M_{\lambda}(b_2)$ is isomorphic to 
$$\text{$\mathbb A^{m_1-m_3-1}$ (resp.\ $\mathbb G_m\times \mathbb A^{m_1-m_3-2}$)}.$$
\end{enumerate}
In all cases, the irreducible components are pairwise disjoint.
\end{theo}

\begin{proof}
We prove the case for $b_1$, and the proof for $b_2$ is similar.
Let $\lambda=(m_1, m_2, m_3)\in X_*(T)_+$ with $X_{\lambda}(b_1)\neq \emptyset$.
The case for $\lambda=(1,0,0)$ follows immediately from the computation in the proof of Lemma \ref{invfinite2}.
So we may assume $\lambda\neq (1,0,0)$.
First note that we have a decomposition 
$$X_{\lambda}(b_1)(\bar k)\cap \eta^{-1}(0)=\bigcup_{(Q,[\mu])\in \mathcal P_{\lambda}(b_1)}((X_{\lambda}(b_1)(\bar k)\cap \eta^{-1}(0))\cap G_{[\mu]}^{C_M}(Q))$$
by Proposition \ref{invdecomposition2}.
It follows from Lemma \ref{invclosed2} that each $(X_{\lambda}(b_1)(\bar k)\cap \eta^{-1}(0))\cap G_{[\mu]}^{C_M}(Q)$ is closed in $X_{\lambda}(b_1)(\bar k)$.
The computation in the proof of Lemma \ref{invfinite2} shows that for each $(Q,[\mu])\in \mathcal P_{\lambda}(b_1)$, the subset $(X_{\lambda}(b_1)(\bar k)\cap \eta^{-1}(0))\cap G_{[\mu]}^{C_M}(Q)$ is equal to 
\begin{align*}
    &\text{$\mathcal{G}r_{[\mu], C_M}^{\mathrm{I\hspace{-.1em}I}}(Q)$,\quad $\mathcal{G}r_{[\mu], C_M}^{\mathrm{I\hspace{-.1em}I\hspace{-.1em}I}}(Q)$\quad or\quad $\mathcal{G}r_{[\mu], D}^{\mathrm{I}}(Q)$} & (m_2=0) \\
    &\mathcal{G}r_{[\mu], D}^{\mathrm{I\hspace{-.1em}I}, -m_2}(Q) & (m_2<0) \\
    &\text{$\mathcal{G}r_{[\mu],b_1}^{\mathrm{I\hspace{-.1em}I},m_2}(Q)$,\quad $\mathcal{G}r_{[\mu],b_1}^{\mathrm{I\hspace{-.1em}I\hspace{-.1em}I},m_2}(Q)$\quad or\quad $\mathcal{G}r_{[\mu],b_1}^{\mathrm{I\hspace{-.1em}I},m_2}(Q)\sqcup \mathcal{G}r_{[\mu],b_1}^{\mathrm{I\hspace{-.1em}I\hspace{-.1em}I},m_2}(Q)$} & (m_2>0)
\end{align*}
($\mathcal{G}r_{[\mu], C_M}^{\mathrm{I\hspace{-.1em}I}}(Q)$ or $\mathcal{G}r_{[\mu], C_M}^{\mathrm{I\hspace{-.1em}I\hspace{-.1em}I}}(Q)$ appears only in the case where $[\mu]=[e_1, e_2, e_3]$ satisfies $e_1>e_2=e_3$ or $e_1=e_2>e_3$).
In the case $(X_{\lambda}(b_1)(\bar k)\cap \eta^{-1}(0))\cap G_{[\mu]}^{C_M}(Q)=\mathcal{G}r_{[\mu],b_1}^{\mathrm{I\hspace{-.1em}I},m_2}(Q)\sqcup \mathcal{G}r_{[\mu],b_1}^{\mathrm{I\hspace{-.1em}I\hspace{-.1em}I},m_2}(Q)$, it is easy to check that both $\mathcal{G}r_{[\mu],b_1}^{\mathrm{I\hspace{-.1em}I},j}(Q)$ and $\mathcal{G}r_{[\mu],b_1}^{\mathrm{I\hspace{-.1em}I\hspace{-.1em}I},j}(Q)$ are also closed in $X_{\lambda}(b_1)(\bar k)$.
So, using Proposition \ref{geometricstructure2}, we have a decomposition of $X_{\lambda}(b_1)$ into irreducible closed subvarieties of dimension $m_1-m_3-1$.
If $m_2\le 0$, then this decomposition is disjoint because $C_M$ and $C_0$ have relative position I (compare the last part of the proof of Theorem \ref{ADLVgeometricstructure}).
If $m_2>0$, then we have $$\mathcal{G}r_{[\mu],b_1}^{\mathrm{I\hspace{-.1em}I},j}(Q)= \mathcal{G}r_{[\mu'],b_1}^{\mathrm{I\hspace{-.1em}I\hspace{-.1em}I},j}(b_1 Q),$$
where $[\mu]=[e_1, e_2, e_3], [\mu']=[e_1, e_2-1, e_3]$.
In this case, we always choose the one for I\hspace{-.1em}I.
We can also check that this gives a disjoint decomposition of $X_{\lambda}(b_1)$ into irreducible components.
Thus all of the irreducible components are parameterized by the elements in $(J_{b_1}/H_{b_1})\times M_{\lambda}(b_1)'$.
Again by Proposition \ref{geometricstructure2}, their geometric structures are given as above.
\end{proof}

\begin{rema}
\label{nonemptiness}
The criterion for non-emptiness of $X_{\lambda}(b)$ is already known (cf.\  \cite[Theorem 4.16]{Gortz}).
Let $\lambda=(m_1, m_2, m_3)\in X_*(T)_+$.
Then
\begin{align*}
X_{\lambda}(1)\neq \emptyset &\Leftrightarrow m_1\geq 0, m_1+m_2\geq 0, m_1+m_2+m_3=0, \\
X_{\lambda}(b_1)\neq \emptyset &\Leftrightarrow m_1\geq \frac{1}{3}, m_1+m_2\geq \frac{2}{3}, m_1+m_2+m_3=1, \\
X_{\lambda}(b_2)\neq \emptyset &\Leftrightarrow m_1\geq \frac{2}{3}, m_1+m_2\geq \frac{4}{3}, m_1+m_2+m_3=2.
\end{align*}
\end{rema}

A flat morphism of $f\colon X\rightarrow Y$ of varieties over $\bar k$ is called an $\mathbb A^n$-fibration, for some integer $n$, if for every $y\in Y$, the fiber $f^{-1}(y)$ is isomorphic to $\mathbb A^n$.

We will now consider an $\mathbb A^n$-fibration over a Deligne-Lusztig variety:
Let $G$ be a connected reductive group defined over $\mathbb F_q$, and let $F\colon G\rightarrow G$ be a Frobenius map over $\mathbb F_q$.
Fix an $F$-stable Borel subgroup $B_0$ containing an $F$-stable maximal torus $T_0$ (such a pair always exists; see \cite[1.17]{Carter}).
Let $W$ be the Weyl group of $T_0$.
For any $w\in W$, we denote by $X(w)$ the Deligne-Lusztig variety associated with $w$.

\begin{coro}
\label{ADLVgeometricstructure3}
Let $\lambda=(m_1, m_2, m_3)\in X_*(T)_+$ with $X_{\lambda}(b)\neq \emptyset$.
\begin{enumerate}[(i)]
\item Let $b=1$.
Then every irreducible component of $X_{\lambda}(1)$ is an $\mathbb A^n$-fibration over a Deligne-Lusztig variety for some $n$ if and only if $\lambda$ is one of the following forms:
$$\lambda=(2r, -r, -r), (r, r, -2r), (2r+3, -r-1, -r-2), (r+2, r+1, -2r-3)$$
for some $r\geq 0$.
In the first two cases, we have
\begin{align*}
X_{(0,0,0)}(1)&\cong \bigsqcup_{J_1/K_0}\{pt\} \hspace{2.7cm} \text{and}\\
X_{\lambda}(1)&\cong \bigsqcup_{J_1/K_0}\Omega\times \mathbb A^{3r-2}\ (r>0)
\end{align*}
as $\bar k$-varieties.
In the last two cases, we have
$$X_{\lambda}(1)\cong (\bigsqcup_{J_1/K_1}\Omega\times \mathbb A^{3(r+1)})\sqcup (\bigsqcup_{J_1/K_2}\Omega\times \mathbb A^{3(r+1)})$$
as $\bar k$-varieties.

\item Let $b=b_1$.
Then every irreducible component of $X_{\lambda}(b_1)$ is an $\mathbb A^n$-fibration over a Deligne-Lusztig variety for some $n$ if and only if $\lambda$ is one of the following forms:
$$\lambda=(2r+1, -r, -r), (r+1, r+1, -2r-1), (r+1, r, -2r), (2r+2, -r, -r-1)$$
for some $r\geq 0$.
In the first two cases, we have
$$
X_{\lambda}(b_1)\cong \bigsqcup_{J_{b_1}/H_{b_1}}\mathbb A^{d}
$$
as $\bar k$-varieties, where $d=3r, 3r+1$ respectively.
In the last two cases, we have
$$X_{\lambda}(b_1)\cong (\bigsqcup_{J_{b_1}/H_{b_1}} \mathbb A^d)\sqcup (\bigsqcup_{J_{b_1}/H_{b_1}} \mathbb A^d)$$
as $\bar k$-varieties, where $d=3r, 3r+2$ respectively.

\item Let $b=b_2$.
Then every irreducible component of $X_{\lambda}(b_2)$ is an $\mathbb A^n$-fibration over a Deligne-Lusztig variety for some $n$ if and only if $\lambda$ is one of the following forms:
$$\lambda=(r+1, r+1, -2r), (2r+2, -r, -r), (2r+1, -r+1, -r), (r+2, r+1, -2r-1)$$
for some $r\geq 0$.
In the first two cases, we have
$$
X_{\lambda}(b_2)\cong \bigsqcup_{J_{b_2}/H_{b_2}}\mathbb A^{d}
$$
as $\bar k$-varieties, where $d=3r, 3r+1$ respectively.
In the last two cases, we have
$$X_{\lambda}(b_2)\cong (\bigsqcup_{J_{b_2}/H_{b_2}} \mathbb A^d)\sqcup (\bigsqcup_{J_{b_2}/H_{b_2}} \mathbb A^d)$$
as $\bar k$-varieties, where $d=3r, 3r+2$ respectively.
\end{enumerate}
\end{coro}
\begin{proof}
In every case, one way follows from Theorem \ref{ADLVgeometricstructure} or Theorem \ref{ADLVgeometricstructure2}.
For the converse, it suffices to show that if a variety $V$ over $\bar k$ is an $\mathbb A^n$-fibration over $\mathbb G_m$ or $\mathbb P^2\setminus \mathbb P^2(k)$, then $V$ cannot be an $\mathbb A^n$-fibration over a Deligne-Lusztig variety. 
To show this, we use the $l$-adic cohomology with compact support ($l\neq \mathop{\mathrm{char}} k$).
If a $\bar k$-variety $V$ is an $\mathbb A^n$-fibration over a $\bar k$-variety $X$, then we have $H^q_c(V)\cong H^{q-2n}_c(X)$ as $\overline{\mathbb Q}_l$-vector spaces (cf.\ \cite[5.5, 5.7]{Srinivasan}).
In particular, the Euler characteristics of $V$ and $X$ are equal, i.e., $\chi(V)=\chi(X)$.
So, to complete the proof, we compare the Euler characteristics.

If a variety $V$ over $\bar k$ is an $\mathbb A^n$-fibration over $\mathbb G_m$, then we have 
$$\chi(V)=\chi(\mathbb G_m)=0.$$
However, by \cite[Theorem 7.5.1, Theorem 7.7.11]{Carter} or \cite[Theorem 7.1]{DL}, the Euler characteristic of a Deligne-Lusztig variety is nonzero.
So $V$ cannot be an $\mathbb A^n$-fibration over a Deligne-Lusztig variety.

If a variety $V$ over $\bar k$ is an $\mathbb A^n$-fibration over $\mathbb P^2\setminus \mathbb P^2(k)$, then we have 
\[
  {\mathop{\mathrm{dim}}}_{\overline{\mathbb Q}_l} H^q_c(V)={\mathop{\mathrm{dim}}}_{\overline{\mathbb Q}_l}H^{q-2n}_c(\mathbb P^2\setminus \mathbb P^2(k))=
  \begin{cases}
    1 & (q=2n+2, 2n+4) \\
    q^2+q & (q=2n+1) \\
    0 & (\text{otherwise}) \\
  \end{cases}
\]
and $$|\chi(V)|=|\chi(\mathbb P^2\setminus \mathbb P^2(k))|=q^2+q-2.$$
If $q=2$, then $|\chi(V)|=4$.
However, by \cite[Theorem 7.5.1, Theorem 7.7.11]{Carter}, the Euler characteristic of a Deligne-Lusztig variety is odd if $q=2$, and different from $\chi(V)$.
So we may assume $q\geq 3$.
We will compare the absolute value of the Euler characteristic of a Deligne-Lusztig variety and $|\chi(V)|=q^2+q-2$.

Let $G, B_0, T_0, W$ be as above.
Let $N$ be the number of positive roots, and let $l$ be the number of simple roots.
Further, let $T$ be an $F$-stable maximal torus of $G$ obtained from a maximally split torus $T_0$ by twisting by $w\in W$.
Thus $T=gT_0g^{-1}$, where $g^{-1}F(g)=\dot{w}$ ($\dot{w}$ is a representative of $w$).
Then, by \cite[Theorem 7.5.1, Theorem 7.7.11]{Carter}, we have
$$|\chi(X(w))|=\frac{|G^F|}{q^{N}|T^F|}.$$
Moreover, by \cite[Corollaire 3.3.22]{DMR}, we have
$$\text{$H^q_c(X(w))=0$ for $0\le q< l(w)$},$$
where $l(w)$ denotes the length of $w\in W$.
Since $\mathop{\mathrm{dim}}X(w)=l(w)$ and $H^q_c(V)\neq 0$ for $q=2n+1, 2n+4$, it is enough to consider the case for $l(w)\geq 3$.
In this case, we obviously have $N\geq l\geq 2$ and $N\geq 3$.

Let us compare $|\chi(X(w))|$ and $|\chi(V)|=q^2+q-2$ under the assumption that $q\geq 3$, $N\geq l\geq 2$ and $N\geq 3$.
Using \cite[Proposition 3.3.7]{Carter}, we may also assume that $G$ is semisimple.
Then, as in the proof of \cite[Proposition 3.6.7]{Carter}, we have
$$|T^F|=(q-\lambda_1)\cdots (q-\lambda_l),$$
where $\lambda_1, \ldots, \lambda_l$ are roots of unity.
For $|G^F|$, as in \cite[2.9]{Carter}, we have
$$|G^F|=q^N\prod_{i=1}^l(q^{d_i}-\epsilon_i),$$
where $d_1+\cdots+d_l=N+l$, $d_1\cdots d_l=|W|$ and $\epsilon_i$ is a root of unity.
So we have
$$|\chi(X(w))|\geq \frac{\prod_{i=1}^l(q^{d_i}-1)}{(q+1)^l}.$$

If $N=l\geq 3$, then $d_1=\cdots=d_l=2$.
So if $q>3$ or $q=3, l> 3$, we have $$|\chi(X(w))|\geq (q-1)^l> (q-1)(q+2).$$
If $q=3$ and $N=l=3$, then $|\chi(V)|=10$ and $|W|=8$.
The facts in \cite[1.18, 2.9]{Carter} imply easily that $|G^F|$ is divisible by $10$ if and only if $|W^F|=4$, where $W^F$ is the subgroup of $F$-stable elements of $W$.
If in this case, we have
$$|G^F|=3^3\cdot 2\cdot 8\cdot(1+3+9+27)=3^3\cdot 4^3\cdot 10.$$
On the other hand, we have $$|T^F|\le (q+1)^l=4^3$$ with equality only if each $\lambda_i=-1$.
However, if $|W^F|=4$, then the equality $|T^F|=4^3$ does not hold.
This can be checked using the techniques in the proof of \cite[Proposition 3.6.7]{Carter} (and the facts on $F:X(T_0)_{\mathbb R}\rightarrow X(T_0)_{\mathbb R}$ and $W^F$ in \cite[1.18]{Carter}).
So $|\chi(X(w))|$ is never equal to $10$.

If $N>l\geq 2$, then it suffices to show $$\frac{\prod_{i=1}^l(q^{d_i}-1)}{(q+1)^l(q-1)(q+2)}>1.$$
Since $d_1+\cdots+d_l=N+l$, $d_1\cdots d_l=|W|$, there exist integers $e_1, e_2$ such that $e_1\geq e_2\geq 1$, $e_1+e_2=N$ (hence $e_1\geq 2$) and
\begin{align*}
\frac{\prod_{i=1}^l(q^{d_i}-1)}{(q+1)^l(q-1)(q+2)}&\geq\left( \frac{q-1}{q+1}\right)^l\frac{q^{e_1}+\cdots+1}{q-1}\frac{q^{e_2}+\cdots+1}{q+2}\\
&> \frac{3^{e_1+e_2}+3^{e_1+e_2-1}+3^{e_1+e_2-2}+3^{e_1+e_2-3}}{5\cdot2^{l+1}}\geq 1.
\end{align*}
This completes the proof.
\end{proof}

\bibliographystyle{myamsplain}
\bibliography{reference}

\providecommand{\bysame}{\leavevmode\hbox to3em{\hrulefill}\thinspace}
\providecommand{\MR}{\relax\ifhmode\unskip\space\fi MR }
\providecommand{\MRhref}[2]{%
  \href{http://www.ams.org/mathscinet-getitem?mr=#1}{#2}
}
\providecommand{\href}[2]{#2}
\begin{thebibliography}{10}

\bibitem{BL}
A.~Beauville and Y.~Laszlo, \emph{Conformal blocks and generalized theta
  functions}, Comm. Math. Phys. \textbf{164} (1994), no.~2, 385--419.
  \MR{1289330}

\bibitem{Carter}
R.~W. Carter, \emph{Finite groups of {L}ie type}, Pure and Applied Mathematics
  (New York), John Wiley \& Sons, Inc., New York, 1985, Conjugacy classes and
  complex characters, A Wiley-Interscience Publication. \MR{794307}

\bibitem{CI}
C.~Chan and A.~B. Ivanov, \emph{Affine {D}eligne-{L}usztig varieties at
  infinite level}, 2018, arXiv:1811.11204.

\bibitem{CV}
M.~Chen and E.~Viehmann, \emph{Affine {D}eligne-{L}usztig varieties and the
  action of {$J$}}, J. Algebraic Geom. \textbf{27} (2018), no.~2, 273--304.
  \MR{3764277}

\bibitem{DL}
P.~Deligne and G.~Lusztig, \emph{Representations of reductive groups over
  finite fields}, Ann. of Math. (2) \textbf{103} (1976), no.~1, 103--161.
  \MR{393266}

\bibitem{DMR}
F.~Digne, J.~Michel, and R.~Rouquier, \emph{Cohomologie des vari\'{e}t\'{e}s de
  {D}eligne-{L}usztig}, Adv. Math. \textbf{209} (2007), no.~2, 749--822.
  \MR{2296313}

\bibitem{Garrett}
P.~Garrett, \emph{Buildings and classical groups}, Chapman \& Hall, London,
  1997. \MR{1449872}

\bibitem{Gashi}
Q.~R. Gashi, \emph{On a conjecture of {K}ottwitz and {R}apoport}, Ann. Sci.
  \'{E}c. Norm. Sup\'{e}r. (4) \textbf{43} (2010), no.~6, 1017--1038.
  \MR{2778454}

\bibitem{Gortz}
U.~G\"{o}rtz, \emph{Affine {S}pringer fibers and affine {D}eligne-{L}usztig
  varieties}, Affine flag manifolds and principal bundles, Trends Math.,
  Birkh\"{a}user/Springer Basel AG, Basel, 2010, pp.~1--50. \MR{3013026}

\bibitem{GHKR}
U.~G\"{o}rtz, T.~J. Haines, R.~E. Kottwitz, and D.~C. Reuman, \emph{Dimensions
  of some affine {D}eligne-{L}usztig varieties}, Ann. Sci. \'{E}cole Norm. Sup.
  (4) \textbf{39} (2006), no.~3, 467--511. \MR{2265676}

\bibitem{GH}
U.~G\"{o}rtz and X.~He, \emph{Basic loci of {C}oxeter type in {S}himura
  varieties}, Camb. J. Math. \textbf{3} (2015), no.~3, 323--353. \MR{3393024}

\bibitem{GH2}
\bysame, \emph{Erratum to: {B}asic loci of {C}oxeter type in {S}himura
  varieties}, Camb. J. Math. \textbf{6} (2018), no.~1, 89--92. \MR{3786099}

\bibitem{GHN}
U.~G\"{o}rtz, X.~He, and S.~Nie, \emph{Fully {H}odge-{N}ewton decomposable
  {S}himura varieties}, Peking Math. J. \textbf{2} (2019), no.~2, 99--154.
  \MR{4060001}

\bibitem{HV}
U.~Hartl and E.~Viehmann, \emph{The {N}ewton stratification on deformations of
  local {$G$}-shtukas}, J. Reine Angew. Math. \textbf{656} (2011), 87--129.
  \MR{2818857}

\bibitem{Kottwitz}
R.~E. Kottwitz, \emph{Orbital integrals on {${\rm GL}\sb{3}$}}, Amer. J. Math.
  \textbf{102} (1980), no.~2, 327--384. \MR{564478}

\bibitem{MV}
I.~Mirkovi\'{c} and K.~Vilonen, \emph{Geometric {L}anglands duality and
  representations of algebraic groups over commutative rings}, Ann. of Math.
  (2) \textbf{166} (2007), no.~1, 95--143. \MR{2342692}

\bibitem{NP}
B.~C. Ng\^{o} and P.~Polo, \emph{R\'{e}solutions de {D}emazure affines et
  formule de {C}asselman-{S}halika g\'{e}om\'{e}trique}, J. Algebraic Geom.
  \textbf{10} (2001), no.~3, 515--547. \MR{1832331}

\bibitem{Nie}
S.~Nie, \emph{Irreducible components of affine {D}eligne-{L}usztig varieties},
  2021, arXiv:1809.03683.

\bibitem{PR}
G.~Pappas and M.~Rapoport, \emph{Twisted loop groups and their affine flag
  varieties}, Adv. Math. \textbf{219} (2008), no.~1, 118--198, With an appendix
  by T. Haines and Rapoport. \MR{2435422}

\bibitem{Srinivasan}
B.~Srinivasan, \emph{Representations of finite {C}hevalley groups}, Lecture
  Notes in Mathematics, vol. 764, Springer-Verlag, Berlin-New York, 1979, A
  survey. \MR{551499}

\bibitem{Viehmann}
E.~Viehmann, \emph{The dimension of some affine {D}eligne-{L}usztig varieties},
  Ann. Sci. \'{E}cole Norm. Sup. (4) \textbf{39} (2006), no.~3, 513--526.
  \MR{2265677}

\bibitem{ZZ}
R.~Zhou and Y.~Zhu, \emph{Twisted orbital integrals and irreducible components
  of affine {D}eligne-{L}usztig varieties}, Camb. J. Math. \textbf{8} (2020),
  no.~1, 149--241. \MR{4085434}

\end{thebibliography}
\end{document}